\documentclass[11pt,bezier]{article}
\usepackage{amsmath,amssymb,amsfonts,euscript,graphicx,hyperref}
\usepackage[all]{xy}

\newtheorem{prethm}{{\bf Theorem}}

\newenvironment{thm}{\begin{prethm}{\hspace{-0.5
               em}{\bf.}}}{\end{prethm}}

\newtheorem{prepro}{{\bf Theorem}}

\newtheorem{preprop}{{\bf Proposition}}

\newenvironment{prop}{\begin{preprop}{\hspace{-0.5
               em}{\bf.}}}{\end{preprop}}

\newtheorem{precor}{{\bf Corollary}}

\newenvironment{cor}{\begin{precor}{\hspace{-0.5
               em}{\bf.}}}{\end{precor}}

\newtheorem{preconj}{{\bf Conjecture}}

\newtheorem{preremark}{{\bf Remark}}

\newenvironment{remark}{\begin{preremark}\rm{\hspace{-0.5
               em}{\bf.}}}{\end{preremark}}

\newtheorem{preexample}{{\bf Example}}

\newenvironment{example}{\begin{preexample}\rm{\hspace{-0.5
               em}{\bf.}}}{\end{preexample}}

\newtheorem{prelem}{{\bf Lemma}}

\newtheorem{prelam}{{\bf Lemma}}

\newtheorem{preproof}{{\bf Proof.}}

\newenvironment{proof}[1]{\begin{preproof}{\rm
               #1}\hfill{$\Box$}}{\end{preproof}}

\begin{document}
\title{\bf \large   On the $j$-invariant and  the Legendre Representation of Elliptic Curves with Complex Multiplication}
\author{\bf{\large{Khashayar Filom}}\thanks{Department of Mathematical Sciences, Sharif University of Technology, Tehran, Iran\hspace{3cm} E-mail:$\mathsf{filom@mehr.sharif.ir}$  This work is part of author's
 Ph.D. thesis at Sharif University under co-advising of H. Fanai and M. Shahshahani.} 
}
\date{}
\maketitle
\begin{abstract}
{By introducing a class of meromorphic functions with certain ramification structures on $\Bbb{CP}^1$, a new method for the determination of the Legendre representation of elliptic curves with complex multiplication is introduced. These functions reduce the desired representation to the solution of an explicitly given system of polynomial equations and makes no use of the knowledge of Hilbert class polynomial on which standard computations depend. As a byproduct, an algorithm for computing $j(k\tau)$ in terms of $j(\tau)$ is obtained and implemented for $k=2,3$.  Solving the system of equations provides a method for computing certain special values of the  modular function $j:\Bbb{H}\rightarrow\Bbb{C}$.}
\end{abstract}
\noindent
\section{Introduction}
Consider an elliptic curve $\mathcal{E}$ in its Legendre form $y^2=x(x-1)(x-\lambda)$,
 $(\lambda\in\Bbb{C}-\{0,1\})$, 
and let $\tau\in\Bbb{H}$ (where $\Bbb{H}$  denotes the upper half plane) be such that $\mathcal{E}\cong\frac{\Bbb{C}}{\Bbb{Z}+\Bbb{Z}\tau}$. 
Relative to the Legendre representation,
$j$-invariant is the rational function 
$j(\lambda)=256\frac{(\lambda^2-\lambda+1)^3}{(\lambda^2-\lambda)^2}$.
It has also a representation as a rational function of Eisenstein series in $\tau$.  It is
a modular function for ${\rm{SL}}_2 (\Bbb{Z})$ on $\Bbb{H}$ 
whose $q$-expansion, $q={\rm{e}}^{2{\rm{\pi}} {\rm{i}}\tau}$, is
$$j(\tau)=\frac{1}{q}+744+196884q+21493760q^2+864299970q^3+20245856256q^4+\cdots$$
The explicit calculation of $j(\tau)$ involves the evaluation of an elliptic integral
which is in fact an algebraic integer when $\tau$ belongs to an imaginary quadratic field.

The purpose of this note is to show that in the special case where the elliptic curve $\mathcal{E}$
admits a complex multiplication, the corresponding $\lambda$ can in principle be computed explicitly by solving a system of polynomial equations.
The novel idea in this approach was motivated from the theory of dessins d'enfants and the realization of elliptic curves via ``proper dessins" as explained in \cite{cwd} where a morphism of hyperelliptic curves which respects the hyperelliptic involutions (in other words, maps a proper dessin on the first Riemann surface to a proper dessin on another) is decomposed into two components and then recovered by calculating the component on $\Bbb{CP}^1$ via exploiting the theory of Belyi functions on $\Bbb{CP}^1$.
However, the exposition here is independent of \cite{cwd} and makes no explicit use of dessins d'enfants (cf. Remark \ref{Belyi}).
The standard method for deriving the $j$-invariant of an elliptic curve with complex multiplication involves finding the Hilbert class polynomial  (which becomes more difficult as the class number of the corresponding quadratic order increases) and then finding its roots \cite{Cohen}. In our approach, no {\it a priori} knowledge of the Hilbert polynomial (class field) is necessary.

In section 2, we recall the relevant background material and explain the polynomial systems determining $\lambda$.  
Applications to specific examples is given in section 3.  
As another application of this point of view, the relationship between $j$-values of isogenous elliptic curves is investigated in section 4. In particular, $j(k\tau)$ is explicitly calculated in terms of $j(\tau)$ for 
$k=2,~{\rm or}~3$.
A table explicitly exhibiting some special elliptic curves, their $j$-invariants, 
their endomorphism rings (or orders in an imaginary quadratic field),
and periods $\tau$, is given in the final section.

\section{Main Strategy} 

To illustrate the main idea, fix a natural number $D$ which is not a perfect square. We are interested in elliptic curves $\mathcal{E}$ for which there exists a degree $D$ isogeny $f:\mathcal{E}\rightarrow\mathcal{E}$.
Considering $\mathcal{E}$ in the form of $\frac{\Bbb{C}}{\Bbb{Z}+\Bbb{Z}\tau}$ where $\tau$ lies in the upper half plane, the existence of such an $f$ puts certain constraints on $\tau$: 

\begin{prop}\label{classification}
With hypothesis mentioned above, $\tau$ and $f$ should be of the form of 
$$\tau=\frac{1}{2b}\left(u+\sqrt{4D-a^2}{\rm{i}}\right)\quad [z]\mapsto \left[\left(\frac{a-u}{2}+b\tau\right)z\right]$$
 where $u,b,a$ are integers with $\mid a\mid <2\sqrt{D}$ and $4b|u^2+4D-a^2$.
\end{prop}

\noindent Proof is not difficult. Furthermore, in such a situation the ring ${\rm{End}}(\mathcal{E})$ of endomorphisms of this elliptic curve is an order in the imaginary quadratic field $\Bbb{Q}\left(\sqrt{a^2-4D}\right)$ 
whose conductor divides $v$ where $v$ is the largest natural number with $4v^2 |4D-a^2$ for even $a$'s or the largest odd number with $v^2 |4D-a^2$ for odd $a$'s. 
This result may be particularly useful for determining the Hilbert class field of an imaginary quadratic field $K$ if our goal of calculating the value of $j$-invariant at the points $\tau=\frac{1}{2b}\left(u+\sqrt{4D-a^2}{\rm{i}}\right)$ of the upper half plane appeared before is achieved: 
describe $K$ as $\Bbb{Q}\left(\sqrt{a^2-4D}\right)$ where $D\in\Bbb{N}$ is not a perfect square and $4D-a^2>0$ is not divisible by any perfect square exceeding four. Then adjoining any $j\left(\tau=\frac{1}{2b}\left(u+\sqrt{4D-a^2}{\rm{i}}\right)\right)$, where as in Proposition \ref{classification}  $4b\mid u^2+4D-a^2$, to $K$ 
generates its Hilbert class field  since in these cases $\mathcal{E}=\frac{\Bbb{C}}{\Bbb{Z}+\Bbb{Z}\tau}$ has complex multiplication by $\mathcal{O}_K$. 

Also fixing $a,D$ in \ref{classification}, set of $\tau$'s in the form of $\tau=\frac{1}{2b}\left(u+\sqrt{4D-a^2}{\rm{i}}\right)$ with $4b|u^2+4D-a^2$ is invariant under the action of ${\rm{SL}}_2(\Bbb{Z})$ since one can easily check that for any 
$\begin{bmatrix}
\alpha &\beta\\
\gamma &\delta
\end{bmatrix}$
 in this group, assuming $\tau^{\prime}=\frac{\alpha\tau+\beta}{\gamma\tau+\delta}$:
$$\tau^{\prime}=\frac{1}{2b^{\prime}}\left(u^{\prime}+\sqrt{4D-{a^2}}{\rm{i}}\right) \quad 4b^{\prime}|{u^{\prime}}^2+4D-a^2,b^{\prime}=b.{\rm{\bf{N}}}(\gamma\tau+\delta)$$
where ${\rm{\bf{N}}}$ denotes the norm in the quadratic number field $\Bbb{Q}\left(\sqrt{a^2-4D}\right)$. This gives us a way to prove that the number of the ${\rm{SL}}_2(\Bbb{Z})$-orbits of these points is finite.  In fact, we can associate with each $\tau=\frac{1}{2b}\left(u+\sqrt{4D-a^2}{\rm{i}}\right)$ the binary quadratic form
\begin{equation*}
(x,y)\mapsto b.{\rm{\bf{N}}}(x+\tau y)=bx^2+uxy+\frac{u^2+4D-a^2}{4b}y^2
\end{equation*}
which has integer coefficients and discriminant $a^2-4D$. The action of ${\rm{SL}}_2(\Bbb{Z})$ on $\tau$'s in Proposition \ref{classification} (for fixed $D,a$) translates to its action on the corresponding binary quadratic forms. 
Therefore number of orbits in this action is finite since number of  ${\rm{SL}}_2(\Bbb{Z})$-equivalence classes of binary quadratic forms with discriminant $a^2-4D$ is finite using a classic theorem of number theory that relates the  equivalence classes of binary quadratic forms with prescribed discriminant to the class group of the corresponding quadratic number field.
So having knowledge of the class group of the quadratic field in consideration, one may list all of these orbits explicitly by exhibiting representatives. We must point out that in fact we are dealing with representations $\frac{1}{2b}\left(u+\sqrt{4D-a^2}{\rm{i}}\right)$ of $\tau$ rather than solely this number. 
Therefore number of  ${\rm{SL}}_2(\Bbb{Z})$-orbits of expressions  $\tau=\frac{1}{2b}\left(u+\sqrt{4D-a^2}{\rm{i}}\right)$ for some prescribed $D$ and $a$ may differ from the class number of $\Bbb{Q}\left({\sqrt{a^2-4D}}\right)$ because there are cases where the associated binary forms are not 
primitive\footnote{For example, let $D=3,a=0$. 
Points $\sqrt{3}{\rm{i}}=\frac{1}{2}\left(\sqrt{12}{\rm{i}}\right)$ where $b=1,u=0$ and $\frac{1+\sqrt{3}{\rm{i}}}{2}=\frac{1}{2\times 2}\left(2+\sqrt{12}{\rm{i}}\right)$ where $b=u=2$ are not equivalent under ${\rm{SL}}_2(\Bbb{Z})$'s action although $\Bbb{Q}\left(\sqrt{-3}\right)$ admits unique factorization.}. 
For instance, when an odd prime divisor $p$ of $u,b$ satisfies $p^2\mid 4D-a^2$ or when $2\mid u,b,a$ and $\frac{ (\frac{u}{2})^2+D-(\frac{a}{2})^2}{b}$ is even.  In summary:

\begin{prop} \label{finiteness}
For a fixed non-square $D\in\Bbb{N}$, the number of  isomorphism classes of endomorphisms $f:\mathcal{E}\rightarrow\mathcal{E}$ where $\mathcal{E}$ is an elliptic curve and ${\rm{deg}}f=D$ is finite. Equivalently, number of ${\rm{SL}}_2(\Bbb{Z})$-orbits of points  $\tau=\frac{1}{2b}\left(u+\sqrt{4D-a^2}{\rm{i}}\right)$ of the upper half plane with $u,b,a\in\Bbb{Z}$,$\mid a\mid<2\sqrt{D}$ and $4b|u^2+4D-a^2$   is finite.  
\end{prop}

There is also another point of view which allows us to calculate  $\lambda$'s corresponding to $\tau$'s mentioned here and therefore $j(\tau)$: 
writing $\mathcal{E}$ in its Legendre form $y^2=x(x-1)(x-\lambda)$, the degree $D$ endomorphism $f:\mathcal{E}\rightarrow\mathcal{E}$ is simply a degree $D$ unramified cover of the underlying compact Riemann surface of genus $1$ with a fixed point that respects the hyperelliptic involution.
Without loss of generality, we may assume that a fixed point of degree $D$ covering $f$ of $y^2=x(x-1)(x-\lambda)$ is the point at infinity where $x,y\to\infty$. This together with the fact that $f$ respects the hyperelliptic involution $(x,y)\mapsto(x,-y)$, imply that $f:\left\{y^2=x(x-1)(x-\lambda)\right\}\rightarrow\left\{y^2=x(x-1)(x-\lambda)\right\}$ can be written as $(x,y)\mapsto\big(h(x),g(x)y\big)$ where $h,g$ 
are meromorphic function on $\Bbb{CP}^1$ with $h(\infty)=\infty$. 
Considering ramifications of maps  in the commutative diagram:
%\begin{figure}[ht]
%\centering
%\includegraphics{diagram1.png}
%\end{figure}
\begin{equation*} \tag{$\star$}
\xymatrix{\mathcal{E}  \ar[r]^f \ar[d]  & \mathcal{E} \ar[d]\\
                 \Bbb{CP}^1 \ar[r]^h            &  \Bbb{CP}^1 }
\end{equation*}
where top row is a degree $D$ unramified covering  and columns are ramified 2-fold covers with critical values $0,1,\lambda,\infty$, we deduce the following constraint on $h$ and hence on $\lambda$:
\begin{prop}\label{constraint}
With same notations, $\lambda\in\Bbb{C}-\{0,1\}$ and $h:\Bbb{CP}^1\rightarrow\Bbb{CP}^1$ is a degree $D$ meromorphic function that satisfies:
\begin{itemize}
\item $h(\infty)=\infty$
\item Multiplicity of $h$ at each critical point is $2$ (hence there are $2D-2$ critical points).
\item Critical values of $h$ belong to the set $\{0,1,\lambda,\infty\}$.
\item $h$ has simple points $0,1,\lambda,\infty$ whose value at any of them  lies in $\{0,1,\lambda,\infty\}$.
\end{itemize}
\end{prop}

Fixing a non-square $D\in\Bbb{N}$, our objective is to study those $\lambda$ and meromorphic functions $h(x)$ on Riemann sphere for which the conditions mentioned in Proposition \ref{constraint} hold so that possible choices for $j(\lambda)=256\frac{(\lambda^2-\lambda+1)^3}{(\lambda^2-\lambda)^2}$   can be determined. To obtain $h(x)$ and $\lambda$, one should solve a rather complicated polynomial system whose unknowns include $\lambda$ and roots and poles of $h$. $j(\lambda)$'s calculated this way will represent all possible values for $j(\tau)$ where $\tau$ is related to $D$ as in Proposition \ref{classification}. In other words, via studying meromorphic functions with such properties, there will be a list including all finitely many values which $j$-invariant achieves on the set of $\tau$'s in the upper
 half plane associated with $D$ by \ref{classification}. Next, using truncation of the $q$-expansion of $j:\Bbb{H}\rightarrow\Bbb{C}$, it is an easy task to determine each $\tau$ corresponds to which $j(\lambda)$ in our list. This is the main strategy we are going to pursue in this paper. 

Before writing down these systems in their general form, it is worth noting two remarks that might serve to simplify calculations. First, since the map $f$ from the elliptic curve $y^2=x(x-1)(x-\lambda)$ to itself is
given by $(x,y)\mapsto\big(h(x),g(x)y\big)$, in $\Bbb{C}(x)$ the element $\frac{h(x)(h(x)-1)(h(x)-\lambda)}{x(x-1)(x-\lambda)}$  is a perfect square, i.e. $g(x)^2$, which is another constraint on $h(x)$ and $\lambda$. 
Secondly, there are symmetries for pairs $\big(h(x),\lambda\big)$ described in Proposition \ref{constraint} that make computations shorter. 
Note that if the pair $\big(h(x),\lambda\big)$ satisfies Proposition\ref{constraint}, then so are pairs 
$\big(1-h(1-x),1-\lambda\big)$ and $\big(\frac{1}{\lambda},\frac{h(\lambda x)}{\lambda}\big)$ derived from conjugating $h(x)$ with M\"obius  
transformations $x\mapsto1-x$ and $x\mapsto\lambda x$ respectively. 
Former fixes infinity and interchanges $0,1$ while latter fixes $0,\infty$ and maps $1$ to $\lambda$. $\lambda\mapsto 1-\lambda$
 and $\lambda\mapsto\frac{1}{\lambda}$ generate the action of the symmetric group $S_3$ on $\Bbb{C}-\{0,1\}$ 
that leaves $j(\lambda)$ invariant and hence changing $h(x)$ with these transformations if necessary, one may permute the roles of $0,1,\lambda$ arbitrarily.

\begin{example}\label{D=2}
We first start with the simplest case where $D=2$. In Proposition \ref{classification}, $a\in\{0,1,2\}$ which yields $\tau$'s belonging to fields $\Bbb{Q}\left(\sqrt{-1}\right)$,$\Bbb{Q}\left(\sqrt{-2}\right)$ and $\Bbb{Q}\left(\sqrt{-7}\right)$ all with class number one. Hence there are three distinct orbits for $D=2$, those of ${\rm{i}},\sqrt{2}{\rm{i}},\frac{1+\sqrt{7}{\rm{i}}}{2}$. In order to find the corresponding $\lambda$'s, it is more judicious to apply the previously appeared constraint rather than the constraints in Proposition \ref{constraint}: $h(x)=\frac{P(x)}{Q(x)}\in\Bbb{C}(x)$ where ${\rm{deg}} Q<{\rm{deg}} P=2$ along with $\lambda\neq 0,1$ are such that $\frac{h(x)(h(x)-1)(h(x)-\lambda)}{x(x-1)(x-\lambda)}=\frac{P(x)(P(x)-Q(x))(P(x)-\lambda Q(x))}{x(x-1)(x-\lambda)Q(x)^3}$   is a perfect square in $\Bbb{C}(x)$. Conditioning on $Q(x)$, a straightforward argument shows that it coincides with either $x$, $x-1$ or $x-\lambda$ up to a constant. With a little effort, in each case  $h(x)$ and $\lambda$ will be determined exploiting the fact that $\frac{P(x)(P(x)-Q(x))(P(x)-\lambda Q(x))}{x(x-1)(x-\lambda)Q(x)^3}$ is a square in $\Bbb{C}(x)$:
\begin{equation*}
\begin{split}
&{\rm{\bf{I.}}}\,\lambda=-1,h(x)=\pm\frac{i(x^2-1)}{2x}\qquad {\rm{\bf{II.}}}\,\lambda=3\pm2\sqrt{2},h(x)=-\frac{1}{2}\frac{(x-1)(x-3\mp2\sqrt{2})}{x}\\
&{\rm{\bf{III.}}}\,\lambda=\frac{1\pm3\sqrt{7}{\rm{i}}}{2},h(x)=\frac{\frac{-3\mp\sqrt{7}i}{8}(x-1)(x-\frac{1\pm 3\sqrt{7}{\rm{i}}}{2})+x}{x}
\end{split}
\end{equation*}
\begin{equation*}
\begin{split}
&{\rm{\bf{IV.}}}\,\frac{1}{\lambda},\frac{h(\lambda x)}{\lambda}\,\text{for}\,\lambda,h(x)\,\text{in}\,{\rm{\bf{III}}}\qquad {\rm{\bf{V.}}}\,1-\lambda,1-h(1-x)\,\text{for}\,\lambda,h(x)\,\text{in}\,{\rm{\bf{I}}},{\rm{\bf{II}}},{\rm{\bf{III}}},{\rm{\bf{IV}}}\\
&{\rm{\bf{VI.}}}\,\frac{1}{\lambda},\frac{h(\lambda x)}{\lambda}\,\text{for}\,\lambda,h(x)\,\text{in}\,{\rm{\bf{V}}}
\end{split}
\end{equation*}
 Hence the value of the modular function $j$ at any of points ${\rm{i}},\sqrt{2}{\rm{i}},\frac{1+\sqrt{7}{\rm{i}}}{2}$ of the upper half plane coincides with the value of $j(\lambda)=256\frac{(\lambda^2-\lambda+1)^3}{(\lambda^2-\lambda)^2}$ at one of points $\lambda=-1,3\pm2\sqrt{2},\frac{1\pm3\sqrt{7}{\rm{i}}}{2}$. Latter values are $1728,20^3,-15^3$. Thus aside from $\tau={\rm{i}}$ which corresponds to the square lattice and the elliptic curve $y^2=x^3-x$ ($\lambda=-1$) where $j=1728$, one observes that $\left\{j\left(\sqrt{2}{\rm{i}}\right),j\left(\frac{1+\sqrt{7}{\rm{i}}}{2}\right)\right\}=\{20^3,-15^3\}$. Applying $q$-expansion to distinguish them, with help of a numerical software:
\begin{equation*}
\begin{split}
&j\left(\sqrt{2}{\rm{i}}\right)\approx {\rm{e}}^{2{\rm{\pi}}\sqrt{2}}+744+196884 {\rm{e}}^{-2{\rm{\pi}}\sqrt{2}}+21493760 {\rm{e}}^{-4{\rm{\pi}}\sqrt{2}}\approx7999.997704\\
&j\left(\frac{1+\sqrt{7}{\rm{i}}}{2}\right)\approx- {\rm{e}}^{{\rm{\pi}}\sqrt{7}}+744-196884 {\rm{e}}^{{\rm{\pi}}\sqrt{7}}+21493760 {\rm{e}}^{-2{\rm{\pi}}\sqrt{7}}-864299970 {\rm{e}}^{-3{\rm{\pi}}\sqrt{7}}\\
&\approx-3375.000073
\end{split}
\end{equation*}
Comparing with $20^3=8000$ and $-15^3=-3375$, we conclude that $j\left(\sqrt{2}{\rm{i}}\right)=20^3$ and $j\left(\frac{1+\sqrt{7}{\rm{i}}}{2}\right)=-15^3$.

\end{example}

Next, we will try to write a general system of polynomial equations whose solutions provide a  list of possible values of $j$ at those $\tau$'s that were associated with our fixed $D>2$ as in 
Proposition \ref{classification}. Let the degree $D$ meromorphic function $h$ on $\Bbb{CP}^1$ satisfy constraints listed in Proposition \ref{constraint}. 
It has $2D-2$ critical points due to Riemann-Hurwitz formula, each with multiplicity two, and at most four critical values, namely $0,1,\lambda,\infty$. 
Above any of them there are at most $\lfloor\frac{D}{2}\rfloor$ critical points. 
We conclude that all of them are critical values as otherwise there will be at most $3\lfloor\frac{D}{2}\rfloor$ critical points. $3\lfloor\frac{D}{2}\rfloor<2D-2$  for any $D\in\Bbb{N}$ except for $D=1,4$ (excluded because they are perfect squares.) and $D=2$ (which was studied separately in Example \ref{D=2}). 
Thus critical values of $h:\Bbb{CP}^1\rightarrow\Bbb{CP}^1$ are precisely $0,1,\lambda,\infty$. 
Moreover, fiber above each of them contains at most $\lfloor\frac{D}{2}\rfloor$ critical points while on the other hand, number of critical points of $h$ were $2D-2$. \\
When $D$ is odd, $4\lfloor\frac{D}{2}\rfloor=2D-2$ and therefore above each critical value $0, 1, \lambda$ or infinity $h$ has one simple point and $\frac{D-1}{2}$ critical points all of multiplicity two. 
From Proposition \ref{constraint}, the simple point $\infty$ is mapped to $\infty$ and furthermore $0, 1, \lambda$  are simple points in critical fibers. Hence from previous discussion, they lie in distinct critical fibers over critical values $0, 1, \lambda$, that is, for a permutation $\sigma$ of $\{0, 1, \lambda\}$ : $h(x)=\sigma(x)  \quad\forall x\in\{0,1,\lambda\}$. 
There are three different situations: $\sigma$ is either identity, or a transposition or a three cycle. Because of aforementioned symmetries between $0,1,\lambda$,  
two former cases reduce to $\sigma(0)=0,\sigma(1)=\lambda,\sigma(\lambda)=1$ and $\sigma(0)=1,\sigma(1)=\lambda,\sigma(\lambda)=0$ respectively without any change in the value of $j(\lambda)$. 
In conclusion:

\begin{cor}\label{oddcase}
For odd non-square $D$, degree $D$ meromorphic function  $h:\Bbb{CP}^1\rightarrow\Bbb{CP}^1$ in Proposition \ref{constraint} may be assumed to satisfy:
\begin{itemize}
\item Critical values of $h$ are precisely $0,1,\lambda,\infty$ over each $h$ possesses a simple point and $\frac{D-1}{2}$ points of multiplicity two.
\item $h(\infty)=\infty$ and either
$\begin{cases}
h(0)=0\\
h(1)=1\\
h(\lambda)=\lambda
\end{cases}$
or
$\begin{cases}
h(0)=0\\
h(1)=\lambda\\
h(\lambda)=1
\end{cases}$
or
$\begin{cases}
h(0)=1\\
h(1)=\lambda\\
h(\lambda)=0
\end{cases}$.
\end{itemize}
\end{cor}

\noindent When $D$ is even, there is at most $\lfloor\frac{D}{2}\rfloor=\frac{D}{2}$ 
critical points in any of the four critical fibers. 
According to conditions in Proposition \ref{constraint}, the fiber above the critical point $\infty$ contains the simple point $\infty$ which 
yields the upper bound $\frac{3D}{2}+(\frac{D}{2}-1)=2D-1$ for the number of critical points of $h$ albeit this number is actually $2D-2$. 
We deduce that there are only two possibilities: either $h$ has $\frac{D}{2}-2$ critical points over $\infty$ whereas number of its critical points over any of critical values $0,1,\lambda$ 
achieves the maximum $\frac{D}{2}$  or $h$ has $\frac{D}{2}-1$ critical points in fiber above $\infty$ and also in fiber above one of $0,1,\lambda$ and possesses $\frac{D}{2}$ critical points over any of remaining two values. 
When the latter possibility occurs, meromorphic function $h:\Bbb{CP}^1\rightarrow\Bbb{CP}^1$ of even degree $D$ has four critical values $0,1,\lambda,\infty$ where for two of them, including $\infty$, over each $h$ has $\frac{D}{2}-1$ critical points and two simple points and
furthermore $\frac{D}{2}$ critical points above any of two remaining critical values in the set $\{0,1,\lambda,\infty\}$. 
Of course all critical points are of multiplicity two. Employing aforementioned symmetries, we may assume that the critical value other than $\infty$ whose fiber contains a simple point is $0$. 
Constraints in Proposition \ref{constraint} imply that points  $0,1,\lambda,\infty$ constitute four simple points in critical fibers $h^{-1} (0),h^{-1} (\infty)$ and besides, 
$\infty\in h^{-1} (\infty)$. Therefore either $0,\lambda\in h^{-1} (0),1\in h^{-1} (\infty)$ or $0,1\in h^{-1} (0),\lambda\in h^{-1} (\infty)$ or $1,\lambda\in h^{-1} (0),0\in h^{-1} (\infty)$. 
Replacing $\big(h(x),\lambda\big)$ with $\big(\frac{h(\lambda x)}{\lambda},\frac{1}{\lambda}\big)$, first case transforms to second one.  
We conclude that there are only three different cases for even $D$'s which are outlined below:

\begin{cor}\label{evencase}
For even non-square $D>2$, degree $D$ meromorphic function  $h:\Bbb{CP}^1\rightarrow\Bbb{CP}^1$ in Proposition \ref{constraint} may be assumed to satisfy:
\begin{itemize}
\item Critical values of $h$ are precisely $0,1,\lambda,\infty$. $h$ either possesses one simple point and $\frac{D}{2}-1$ points of multiplicity two over each of $0,\infty$ and $\frac{D}{2}$ points of multiplicity two over each of $1,\lambda$ or it has $\frac{D}{2}$ points of multiplicity two over each of $0,1,\lambda$ and four simple points over $\infty$. 
\item Either
$h(0)=h(1)=h(\lambda)=h(\infty)=\infty$
or 
$\begin{cases}
h(1)=h(\lambda)=0\\
h(0)=h(\infty)=\infty
\end{cases}$\\
or
$\begin{cases}
h(0)=h(1)=0\\
h(\lambda)=h(\infty)=\infty
\end{cases}$.
\end{itemize}
\end{cor}
 For any of three possibilities outlined in either of Corollaries \ref{oddcase} or \ref{evencase} for odd and even $D$'s respectively, it is easy to verify $\frac{h(x)(h(x)-1)(h(x)-\lambda)}{x(x-1)(x-\lambda)}$ is square of another meromorphic function. 
 Thus exhibiting such a pair $\big(h(x),\lambda\big)$ and then computing $j(\lambda)=256\frac{(\lambda^2-\lambda+1)^3}{(\lambda^2-\lambda)^2}$, always leads to a value of the  modular $j$-function at one point from our finite list of representatives of those $\tau$'s in the upper half plane which were associated with $D$ in Proposition \ref{classification}. 
In each of the six cases that appeared in Corollaries \ref{oddcase} and \ref{evencase}, in order to determine 
$\big(h(x),\lambda\big)$ 
there is a polynomial system of equations to solve that we will introduce below and everything reduces to finding their solutions, a task that is the main part of our computational method.

We shall explain the way to form the systems corresponding to these cases for the first set of conditions in Corollary \ref{oddcase} (i.e. when $h(0)=0,h(1)=1,h(\lambda)=\lambda$) and the remaining five cases can be dealt with similarly. 
Here $h$ is a degree $D$ meromorphic function on Riemann sphere with one simple zero and one simple pole at $0$ and $\infty$ respectively and $\frac{D-1}{2}$ zeros and $\frac{D-1}{2}$ poles each of multiplicity two. 
So $h$ can be written as
$h(x)=k\frac{x\prod_{i=1}^{\frac{D-1}{2}}(x-\alpha_i)^2}{\prod_{i=1}^{\frac{D-1}{2}}(x-\beta_i)^2}$.
Also only simple roots of numerators of $h(x)-1$ and $h(x)-\lambda$ are $1$ and $\lambda$ respectively with other $ \frac{D-1}{2}$ roots all with multiplicity two.  
The fact that forces them to be in the following forms:
$$h(x)-1=k\frac{(x-1)\prod_{i=1}^{\frac{D-1}{2}}(x-\gamma_i)^2}{\prod_{i=1}^{\frac{D-1}{2}}(x-\beta_i)^2}\quad h(x)-\lambda=k\frac{(x-\lambda)\prod_{i=1}^{\frac{D-1}{2}}(x-\delta_i)^2}{\prod_{i=1}^{\frac{D-1}{2}}(x-\beta_i)^2}$$
These identities imply $kx\prod_{i=1}^{\frac{D-1}{2}}(x-\alpha_i)^2-\prod_{i=1}^{\frac{D-1}{2}}(x-\beta_i)^2=k(x-1)\prod_{i=1}^{\frac{D-1}{2}}(x-\gamma_i)^2$ and $kx\prod_{i=1}^{\frac{D-1}{2}}(x-\alpha_i)^2-\lambda\prod_{i=1}^{\frac{D-1}{2}}(x-\beta_i)^2=k(x-\lambda)\prod_{i=1}^{\frac{D-1}{2}}(x-\delta_i)^2$. 
The equality of coefficients both sides of the equations yields a system consisting of $2D$ polynomial equations and $2+4(\frac{D-1}{2})=2D$ unknowns $k,\lambda,\alpha_i,\beta_i,\gamma_i,\delta_i\,(1\leq i\leq \frac{D-1}{2})$. 
Note that $k\neq 0$,$\lambda\neq 0,1$ and moreover $\alpha_i,\beta_i,\gamma_i,\delta_i$'s are pairwise distinct as they are different critical points of $h$. 
Using same argument for the two remaining cases in Corollary \ref{oddcase} and also cases mentioned  in Corollary \ref{evencase}, we arrive at our main Theorem:
\begin{thm}\label{systems}
Let $D>2$ be non-square. Consider three systems derived from the following three groups of polynomial identities for odd $D$'s:
\begin{equation}\label{system1}
\begin{cases}
kx\prod_{i=1}^{\frac{D-1}{2}}(x-\alpha_i)^2-\prod_{i=1}^{\frac{D-1}{2}}(x-\beta_i)^2=k(x-1)\prod_{i=1}^{\frac{D-1}{2}}(x-\gamma_i)^2\\
kx\prod_{i=1}^{\frac{D-1}{2}}(x-\alpha_i)^2-\lambda\prod_{i=1}^{\frac{D-1}{2}}(x-\beta_i)^2=k(x-\lambda)\prod_{i=1}^{\frac{D-1}{2}}(x-\delta_i)^2
\end{cases}
\end{equation}
\begin{equation}\label{system2}
\begin{cases}
kx\prod_{i=1}^{\frac{D-1}{2}}(x-\alpha_i)^2-\prod_{i=1}^{\frac{D-1}{2}}(x-\beta_i)^2=k(x-\lambda)\prod_{i=1}^{\frac{D-1}{2}}(x-\gamma_i)^2\\
kx\prod_{i=1}^{\frac{D-1}{2}}(x-\alpha_i)^2-\lambda\prod_{i=1}^{\frac{D-1}{2}}(x-\beta_i)^2=k(x-1)\prod_{i=1}^{\frac{D-1}{2}}(x-\delta_i)^2
\end{cases}
\end{equation}
\begin{equation}\label{system3}
\begin{cases}
k(x-\lambda)\prod_{i=1}^{\frac{D-1}{2}}(x-\alpha_i)^2-\prod_{i=1}^{\frac{D-1}{2}}(x-\beta_i)^2=kx\prod_{i=1}^{\frac{D-1}{2}}(x-\gamma_i)^2\\
k(x-\lambda)\prod_{i=1}^{\frac{D-1}{2}}(x-\alpha_i)^2-\lambda\prod_{i=1}^{\frac{D-1}{2}}(x-\beta_i)^2=k(x-1)\prod_{i=1}^{\frac{D-1}{2}}(x-\delta_i)^2
\end{cases}
\end{equation} 
and three systems derived from the following three group of polynomial identities for even $D$'s:  
\begin{equation}\label{system4}
\begin{cases}
k\prod_{i=1}^{\frac{D}{2}}(x-\alpha_i)^2-x(x-1)(x-\lambda)\prod_{i=1}^{\frac{D}{2}-2}(x-\beta_i)^2=k\prod_{i=1}^{\frac{D}{2}}(x-\gamma_i)^2\\
k\prod_{i=1}^{\frac{D}{2}}(x-\alpha_i)^2-\lambda x(x-1)(x-\lambda)\prod_{i=1}^{\frac{D}{2}-2}(x-\beta_i)^2=k\prod_{i=1}^{\frac{D}{2}}(x-\delta_i)^2
\end{cases}
\end{equation}
\begin{equation}\label{system5}
\begin{cases}
k(x-1)(x-\lambda)\prod_{i=1}^{\frac{D}{2}-1}(x-\alpha_i)^2-x\prod_{i=1}^{\frac{D}{2}-1}(x-\beta_i)^2=k\prod_{i=1}^{\frac{D}{2}}(x-\gamma_i)^2\\
k(x-1)(x-\lambda)\prod_{i=1}^{\frac{D}{2}-1}(x-\alpha_i)^2-\lambda x\prod_{i=1}^{\frac{D}{2}-1}(x-\beta_i)^2=k\prod_{i=1}^{\frac{D}{2}}(x-\delta_i)^2
\end{cases}
\end{equation}
\begin{equation}\label{system6}
\begin{cases}
kx(x-1)\prod_{i=1}^{\frac{D}{2}-1}(x-\alpha_i)^2-(x-\lambda)\prod_{i=1}^{\frac{D}{2}-1}(x-\beta_i)^2=k\prod_{i=1}^{\frac{D}{2}}(x-\gamma_i)^2\\
kx(x-1)\prod_{i=1}^{\frac{D}{2}-1}(x-\alpha_i)^2-\lambda(x-\lambda)\prod_{i=1}^{\frac{D}{2}-1}(x-\beta_i)^2=k\prod_{i=1}^{\frac{D}{2}}(x-\delta_i)^2
\end{cases}
\end{equation}
where each system has $2D$ equations and $2D$ unknowns, i.e. $k\neq 0$ along with $\lambda$ ,$\alpha_i$'s,$\beta_i$'s,$\gamma_i$'s, $\delta_i$'s, which are pairwise distinct numbers in $\Bbb{C}-\{0,1\}$. 
The values that $j(\lambda)=256\frac{(\lambda^2-\lambda+1)^3}{(\lambda^2-\lambda)^2}$ achieves for solutions to the three systems corresponding to $D$, 
are exactly the values of the modular function $j:\Bbb{H}\rightarrow\Bbb{C}$ at points of the upper half plane in the form of $\tau=\frac{1}{2b}\left(u+\sqrt{4D-a^2}{\rm{i}}\right)$ in which $u,a,b$ are integers with $\mid a\mid<2\sqrt{D}$ and $4b\mid u^2+4D-a^2$.
\end{thm}

\begin{remark} \label{Belyi}
%Conditions we imposed on number of the critical values of $h:\Bbb{CP}^1\rightarrow\Bbb{CP}^1$ and also the %computations one has to handle in order to determine $h(x)$ via solving a system of polynomial equations %resemble to those of Belyi theory when a Belyi function on Riemann sphere is needed to be recovered from its %dessin.
The above method resembles that of computing a Belyi function on the Riemann sphere from the 
knowledge of the corresponding dessin. 
But the functions $h(x)$ introduced in Proposition \ref{constraint} are not Belyi because they possess four critical values instead of three.
The difference  makes this family of functions more rigid than Belyi functions 
since the number of Belyi functions corresponding to a dessin on $\Bbb{CP}^1$ is finite only after rigidifying the dessin with fixing three of its vertices. Also in the Belyi case, functions have a model over $\bar{\Bbb{Q}}$ whereas in our case $h(x)$ is actually defined over algebraic numbers in the sense of $h(x)\in\bar{\Bbb{Q}}(x)$, cf. Theorem \ref{arithmeticity}.
\end{remark}

\begin{remark}\label{parity}
Sometimes it is possible, even without solving equations, to determine $\lambda$ corresponding to a $\tau=\frac{1}{2b}\left(u+\sqrt{4D-a^2}{\rm{i}}\right)$ (written in the form of Proposition \ref{classification} for a prescribed $D$) should be obtained as a solution to which of the three systems associated with $D$ in Theorem \ref{systems}. 
We should analysis how the corresponding degree $D$ isogeny of $\mathcal{E}=\frac{\Bbb{C}}{\Bbb{Z}+\Bbb{Z}\tau}$, introduced in Proposition \ref{classification} 
as $[z]\mapsto\left[\left(\frac{a-u}{2}+b\tau\right)z\right]$, 
permutes four critical points of the ramified 2-fold cover $\mathcal{E}\rightarrow\Bbb{CP}^1$ 
that appeared in columns of commutative diagram $(\star)$. 
This map is combination of the Weierstrass function $\wp$ of the lattice $\Bbb{Z}+\Bbb{Z}\tau$ with a linear map of $\Bbb{CP}^1$ which bijects the set 
$\big\{\wp(\frac{1}{2}),\wp(\frac{\tau}{2}),\wp(\frac{\tau+1}{2})\big\}$ onto $\{0,1,\lambda\}$. 
So its critical points are $[\frac{1}{2}], [\frac{\tau}{2}], [\frac{1+\tau}{2}]$ along with $[0]$, 
i.e. the identity element of our elliptic curve and also the unique point of $\mathcal{E}\rightarrow\Bbb{CP}^1$ above $\infty$. 
Thus we must study how multiplication by $\frac{a-u}{2}+b\tau$ permutes elements of the set $\{\frac{1}{2},\frac{\tau}{2},\frac{1+\tau}{2}\}$ mod the lattice $\Bbb{Z}+\Bbb{Z}\tau$. 
This might be recovered from the associated quadratic binary form $bx^2+uxy+\frac{u^2+4D-a^2}{4b}y^2$. 
In fact when $D$ is odd, systems \eqref{system1} and \eqref{system3} occur iff coefficients of this form are all even or all odd respectively. 
Otherwise we are dealing with the system \eqref{system2}. 
For even $D$'s, system  \eqref{system4} occurs exactly when all of the coefficients of this form are even.
\end{remark}
\begin{remark}\label{perfectsquare}
Although the proposed method does not seem to be applicable when $D$ is a perfect square, 
a careful analysis of the multiplication by $n\in\Bbb{Z}$  map 
on an elliptic curve $\mathcal{E}=\frac{\Bbb{C}}{\Bbb{Z}+\Bbb{Z}\tau}$, 
shows that even for such $D$'s any solution to one of our systems 
except systems \eqref{system1} and \eqref{system4} leads to a $j$-invariant computation. 
Obviously Proposition \ref{classification} remains true for a perfect square $D$ if we exclude the cases where $f:\mathcal{E}\rightarrow\mathcal{E}$ is the multiplication by $\pm\sqrt{D}$ map. 
To see the reason, consider the commutative diagram $(\star\star)$

%\begin{figure}[ht]
%\centering
%\includegraphics{diagram3.png}
%\end{figure}

\begin{equation*} \tag{$\star\star$}
\xymatrixcolsep{4pc}\xymatrix{\frac{\Bbb{C}}{\Bbb{Z}+\Bbb{Z}\tau}  \ar [r]^{P\mapsto\pm\sqrt{D}P } \ar[d]_\wp  & \frac{\Bbb{C}}{\Bbb{Z}+\Bbb{Z}\tau} \ar [d]_\wp\\ 
                 \Bbb{CP}^1 \ar[r]^h                                                                                                       &  \Bbb{CP}^1 }
\end{equation*}

\noindent whose columns are the Weierstrass elliptic function of lattice $\Bbb{Z}+\Bbb{Z}\tau$. 
The ramification structure of the induced map $h$ -which satisfies conditions in \ref{constraint} after being conjugated with some suitable linear map in order to $0,1,\infty$ be its critical values- 
is different from the ramification structure (outlined in 
Corollaries \ref{oddcase} , \ref{evencase}) of a degree $D$ meromorphic function on $\Bbb{CP}^1$ which solves one of systems \eqref{system2}, \eqref{system3}, \eqref{system5} or \eqref{system6}. 
This is due to the fact that  in the bottom row of this diagram 
critical values of $h:\Bbb{CP}^1\rightarrow\Bbb{CP}^1$  
must lie among those of Weierstrass function, i.e. $\wp(0)=\infty$, $\wp\big(\frac{1}{2}\big)$, $\wp\big(\frac{\tau}{2}\big)$ and $\wp\big(\frac{\tau+1}{2}\big)$, while these  are also simple points of $h$.  
So they constitute all four simple points in critical fibers of $h$. 
But by diagram chasing, $h$ fixes them when $D$ is odd (first case in Corollary \ref{oddcase} 
which corresponds to the system \eqref{system1}) and maps them all to $\infty$ when $D$ is even (first case in 
Corollary \ref{evencase} which corresponds to system \eqref{system4}).
Therefore any solution $\big(h(x),\lambda\big)$ to one of the remaining four systems, provides us with a degree $D$ isogeny $f:(x,y)\mapsto\big(h(x),yg(x)\big)$ of elliptic curve $y^2=x(x-1)(x-\lambda)$ 
(where $g(x)$ is a square root of $\frac{h(x)(h(x)-1)(h(x)-\lambda)}{x(x-1)(x-\lambda)}$ in 
the field $\Bbb{C}(x)$.) which differs from a multiplication map on this elliptic curve because  the ramification structure of its quotient under the hyperelliptic involution $y\mapsto -y$ (which is $h$) is different from that of a multiplication map. 
This guarantees that the elliptic curve $y^2=x(x-1)(x-\lambda)$ has complex multiplication and
the isogeny $f$ along with normalized periods $\tau\in\Bbb{H}$ can be written as in Proposition \ref{classification}.
\end{remark}

We will end this section with a Theorem concerning the nature of solutions of systems that appeared in Theorem \ref{systems}.

\begin{thm}\label{arithmeticity}
For a fixed non-square $D\in\Bbb{N}$, the number of pairs $\big(h(x),\lambda\big)$ consisting of a degree $D$ meromorphic function $h$ on the Riemann sphere and a complex number $\lambda\neq 0,1$ that obey constraints in Proposition \ref{constraint} is finite and all of them are defined over $\bar{\Bbb{Q}}$. Consequently, those three systems among  \eqref{system1} through \eqref{system6} which are associated with $D$ according to its parity have only finite number of solutions all of them lying in $\bar{\Bbb{Q}}$.  
\end{thm}
\begin{proof}
{From Theorem \ref{systems}, for any such a pair $\big(h(x),\lambda\big)$ or equivalently in any solution to one of the systems, $j(\lambda)$ is $j$-invariant of an elliptic curve with complex multiplication. Thus $j(\lambda)$ and therefore $\lambda$ is algebraic. 
This implies that the  meromorphic function \hspace{2cm} $h:\Bbb{CP}^1\rightarrow\Bbb{CP}^1$ satisfying the properties specified in Proposition \ref{constraint} has a model over $\bar{\Bbb{Q}}$ as its orbit under ${\rm{Gal}}\big(\frac{\Bbb{C}}{\Bbb{Q}}\big)$ is finite up to isomorphism: 
Fixing critical values $0,1,\infty,\lambda$, the number of isomorphism classes of degree $D$ meromorphic functions on the Riemann sphere with only these critical values is finite (since they are determined by their monodromy representation). 
Also the orbits of these critical values under the action of ${\rm{Gal}}\big(\frac{\Bbb{C}}{\Bbb{Q}}\big)$ is finite because of $\lambda\in\bar{\Bbb{Q}}$. One simple observation implies 
that $h$ is actually defined over $\bar{\Bbb{Q}}$: fixing a $\big(h_1(x),\lambda_1\big)$, the number of pairs $\big(h_2(x),\lambda_2\big)$ where $h_1,h_2$ are isomorphic as meromorphic functions $\Bbb{CP}^1\rightarrow\Bbb{CP}^1$ does not exceed $\binom{4}{3}^2\mid{\rm{S}}_3\mid$:
there are M\"obius transformations $\beta$ and $\beta^{\prime}$ with $\beta^{\prime}\circ h_1=h_2\circ\beta$. Hence $\beta^{\prime}$ maps the critical values of $h_1$ to those of $h_2$ and $\beta$ takes four simple points over critical values of $h_1$ to similar points for $h_2$. 
We conclude that both of M\"obius transformations $\beta$ and $\beta^{\prime}$ map the set $\{0,1,\lambda_1,\infty\}$ onto $\{0,1,\lambda_2,\infty\}$ and therefore $\lambda_2$ belongs to the orbit of $\lambda_1$ under the action of ${\rm{S}}_3$ on $\Bbb{C}-\{0,1\}$. 
These conditions imply that the number of choices for $\lambda_2$,$\beta$,$\beta^{\prime}$ and thus $h_2$ is finite and in fact we get the desired bound since there are $\mid{\rm{S}}_3\mid$ choices for $\lambda_2$ having $\lambda_1$ in hand and each of these M\"obius transformations should map a prescribed four element set onto another specified four element set. 
So any of them has at most $\binom{4}{3}$ possibilities. 
This finishes the proof of our claim that for all pairs $\big(h(x),\lambda\big)$, which we were interested in,   $h(x)\in\Bbb{C}(x)$ is defined over $\bar{\Bbb{Q}}$. 
Consequently, all of its coefficients, roots and poles and hence solutions to any of the polynomial systems that appeared in Theorem \ref{systems} are algebraic. 
But the subset of solutions is a subvariety of the  $2D$-dimensional complex affine space and therefore forced to be finite. 
As a result, the number of solutions to any of these systems and thus number of pairs $\big(h(x),\lambda\big)$ described in \ref{constraint} is finite.}
\end{proof}
From the above finiteness theorem, finiteness of number of isomorphism classes of degree $D$ isogenies $f:\mathcal{E}\rightarrow\mathcal{E}$ between elliptic curves can also be inferred, the fact that was previously stated in Proposition \ref{finiteness} too.

\section{Examples}
In this section we mainly concentrate on examples to illustrate our method.
\begin{example}\label{D=3}
Let us apply the machinery developed so far by finishing the calculations for $D=3$. 
One obtains the values of $j$ at the  points $\tau$ of the upper half plane which represent all the 
${\rm{SL}}_2 (\Bbb{Z})$-orbits in Proposition \ref{classification} for $D=3$, 
owing to the fact that any imaginary quadratic number field of the form of $\Bbb{Q}\left(\sqrt{a^2-4\times 3}\right)$ has class number one.  They are
\begin{equation}\label{list1}
\begin{split}
 &\frac{1}{2}\sqrt{4\times 3}{\rm{i}}=\sqrt{3}{\rm{i}}\quad\frac{1}{2}\sqrt{4\times 3-1^2}{\rm{i}}=\frac{1+\sqrt{11}{\rm{i}}}{2}\quad\frac{1}{2}\sqrt{4\times 3-2^2}{\rm{i}}=\sqrt{2}{\rm{i}}\\
&\frac{1}{2}\sqrt{4\times 3-3^2}{\rm{i}}=\frac{1+\sqrt{3}{\rm{i}}}{2}
\end{split}
\end{equation}
Systems \eqref{system1},\eqref{system2},
%\eqref{system3}
 when $D=3$ turn out to be systems with six unknown and six equations each obtained from one of  the sets of two polynomial identities below:
$$\begin{cases}
kx(x-\alpha)^2-(x-\beta)^2=k(x-1)(x-\gamma)^2\\
kx(x-\alpha)^2-\lambda(x-\beta)^2=k(x-\lambda)(x-\delta)^2
\end{cases}
\begin{cases}
kx(x-\alpha)^2-(x-\beta)^2=k(x-\lambda)(x-\gamma)^2\\
kx(x-\alpha)^2-\lambda(x-\beta)^2=k(x-1)(x-\delta)^2
\end{cases}$$
First one is easier to solve which yields $\lambda={\rm{e}}^{\pm{\frac{{\rm{\pi i}}}{3}}}$ where  $j(\lambda)=256\frac{(\lambda^2-\lambda+1)^3}{(\lambda^2-\lambda)^2}$ vanishes. 
It is well-known that the value zero of $j$-invariant corresponds to the hexagonal lattice and therefore the period $\tau=\frac{1+\sqrt{3}{\rm{i}}}{2}$ in the list \eqref{list1} of $\tau$'s. 
Also we arrive at the following set of polynomial identities:
$$\begin{cases}
-\frac{1}{3}x\big(x-(1+{\rm{e}}^{\frac{{\rm{\pi i}}}{3}})\big)^2-\big(x-\frac{1+{\rm{e}}^{\frac{{\rm{\pi i}}}{3}}}{3}\big)^2=-\frac{1}{3}(x-1)(x-{\rm{e}}^{\frac{2{\rm{\pi i}}}{3}})^2\\
-\frac{1}{3}x\big(x-(1+{\rm{e}}^{\frac{{\rm{\pi i}}}{3}})\big)^2-{\rm{e}}^{\frac{{\rm{\pi i}}}{3}}\big(x-\frac{1+{\rm{e}}^{\frac{{\rm{\pi i}}}{3}}}{3}\big)^2=-\frac{1}{3}(x-{\rm{e}}^{\frac{{\rm{\pi i}}}{3}})(x+{\rm{e}}^{\frac{2{\rm{\pi i}}}{3}})^2
\end{cases}$$
where $\lambda={\rm{e}}^{\frac{{\rm{\pi i}}}{3}}$, and also its Galois conjugate below where $\lambda={\rm{e}}^{-\frac{{\rm{\pi i}}}{3}}$:
$$\begin{cases}
-\frac{1}{3}x\big(x-(1+{\rm{e}}^{-\frac{{\rm{\pi i}}}{3}})\big)^2-\big(x-\frac{1+{\rm{e}}^{-\frac{{\rm{\pi i}}}{3}}}{3}\big)^2=-\frac{1}{3}(x-1)(x-{\rm{e}}^{-\frac{2{\rm{\pi i}}}{3}})^2\\
-\frac{1}{3}x\big(x-(1+{\rm{e}}^{-\frac{{\rm{\pi i}}}{3}})\big)^2-{\rm{e}}^{-\frac{{\rm{\pi i}}}{3}}\big(x-\frac{1+{\rm{e}}^{-\frac{{\rm{\pi i}}}{3}}}{3}\big)^2=-\frac{1}{3}(x-{\rm{e}}^{-\frac{{\rm{\pi i}}}{3}})(x+{\rm{e}}^{-\frac{2{\rm{\pi i}}}{3}})^2
\end{cases}$$
Solutions for the second system are the following two sets of polynomial identities:
\begin{equation*}
\begin{split}
&\begin{cases}
-\frac{1}{3}\big(x-\frac{\lambda+1}{2}\big)^2-\big(x-\frac{2\lambda}{\lambda+1}\big)^2=-\frac{1}{3}(x-\lambda)(x+1)^2\\
-\frac{1}{3}\big(x-\frac{\lambda+1}{2}\big)^2-\lambda\big(x-\frac{2\lambda}{\lambda+1}\big)^2=-\frac{1}{3}(x-1)(x+\lambda)^2
\end{cases}\\
&\qquad\qquad\qquad\lambda=-7\pm 4\sqrt{3}
\end{split}
\end{equation*}
\begin{equation*}
\begin{split}
&\begin{cases}
\frac{\lambda(t+1)^2}{(\lambda-t)^2}x\big(x-\frac{1+t\lambda}{1+t}\big)^2-\big(x-\frac{\lambda (1-t)}{\lambda-t}\big)^2=\frac{\lambda(t+1)^2}{(\lambda-t)^2}(x-\lambda)(x-t)^2\\
\frac{\lambda(t+1)^2}{(\lambda-t)^2}x\big(x-\frac{1+t\lambda}{1+t}\big)^2-\lambda\big(x-\frac{\lambda (1-t)}{\lambda-t}\big)^2=\frac{\lambda(t+1)^2}{(\lambda-t)^2}(x-1)(x+t\lambda)^2
\end{cases}\\
&\qquad\qquad\qquad\qquad t\in\{\pm{\rm{i}}\}\quad
\lambda\in\{3\pm 2\sqrt{2}\}
\end{split}
\end{equation*}
that lead to   $j$ values $256\frac{(\lambda^2-\lambda+1)^3}{(\lambda^2-\lambda)^2}\mid_{-7\pm 4\sqrt{3}}=16\times 15^3=54000$ and $256\frac{(\lambda^2-\lambda+1)^3}{(\lambda^2-\lambda)^2}\mid_{3\pm 2\sqrt{2}}=20^3=8000$
\footnote{Here a simple observation was employed stating that the value of  $j(\lambda)=256\frac{(\lambda^2-\lambda+1)^3}{(\lambda^2-\lambda)^2}$   at roots of $\lambda^2+\mu\lambda+1=0$ equals $256\frac{(\mu+1)^3}{\mu+2}$. }.
The latter was obtained and numerically verified as $j\left(\sqrt{2}{\rm{i}}\right)$ (see Example \ref{D=2}). 
For the former, using $q$-expansion and estimating $j$-invariant at the point $\sqrt{3}{\rm{i}}$ of the upper half plane that appeared above in the list \eqref{list1}, we obtain
$$j\left(\sqrt{3}{\rm{i}}\right)\approx {\rm{e}}^{2\pi\sqrt{3}}+744+196884{\rm{e}}^{-2\pi\sqrt{3}}\approx 53999.992414$$
So we deduce that $j\left(\sqrt{3}{\rm{i}}\right)=16\times 15^3$.
Finally, system \eqref{system3} for $D=3$ yields
$$\begin{cases}
k(x-\lambda)(x-\alpha)^2-(x-\beta)^2=kx(x-\gamma)^2\\
k(x-\lambda)(x-\alpha)^2-\lambda(x-\beta)^2=k(x-1)(x-\delta)^2
\end{cases}$$
Its solutions are the following polynomial identities:
$$\begin{cases}
-\frac{16(t+1)^3}{(t+2)(3t+2)^3}(x-\frac{(t+2)^3(3t+2)}{16(t+1)^3})(x-\frac{(t+2)(3t+2)}{4(t+1)})^2-(x-\frac{(t+2)^2}{4(t+1)})^2\\
=-\frac{16(t+1)^3}{(t+2)(3t+2)^3}x(x-\frac{(t+2)(3t+2)}{4(t+1)^2})^2\\

\\
-\frac{16(t+1)^3}{(t+2)(3t+2)^3}(x-\frac{(t+2)^3(3t+2)}{16(t+1)^3})(x-\frac{(t+2)(3t+2)}{4(t+1)})^2-\frac{(t+2)^3(3t+2)}{16(t+1)^3}(x-\frac{(t+2)^2}{4(t+1)})^2\\
=-\frac{16(t+1)^3}{(t+2)(3t+2)^3}(x-1)(x+\frac{t(3t+4)}{4(t+1)^2}-1)^2
\end{cases}$$
where $\lambda=\frac{(t+2)^3(3t+2)}{16(t+1)^3}$ and $t$ is a root of $(3t+2)^4(t+2)^4-16t(t+1)^3(3t+4)^3$, a polynomial that factors as:
$$(3t^2+6t+4)\left[(3t^2+6t+4)^3-8t(3t^2+6t+4)^2-24t^2(3t^2+6t+4)-16t^3 \right] $$
So either $t=-1\pm\frac{\sqrt{3}{\rm{i}}}{3}$ where $\lambda=\frac{1\mp\sqrt{3}{\rm{i}}}{2}$ at which $j(\lambda)$ vanishes, i.e. we are in the case of the hexagonal elliptic curve $y^2=x^3-1$ ($\tau=\frac{1+\sqrt{3}{\rm{i}}}{2}$ 
from the list \eqref{list1}\footnote{It might seem strange that $\tau=\frac{1+\sqrt{3}{\rm{i}}}{2}$ occurs in solutions to systems of both kinds \eqref{system1} and \eqref{system3}. 
This is due to different ways of representing this number in the form described in Proposition \ref{classification}:  $\tau=\frac{1}{2\times 2}\left(2+\sqrt{4\times 3}{\rm{i}}\right)$ and $\tau= \frac{1}{2}\left(1+\sqrt{4\times 3-3^2}{\rm{i}}\right)$. 
The associated binary forms are $2x^2+2xy+2y^2$ and $x^2+xy+y^2$  respectively. 
Cf. Remark \ref{parity}.}), or $t$ satisfies the quadratic equation $3t^2+6t+4=ut$ where $u^3-8u^2-24u-16=0$.
In latter situation,  with the help of a numerical software, one observes that for any of six $t$'s mentioned there, at $\lambda=\frac{(t+2)^3(3t+2)}{16(t+1)^3}$ to six digits after decimal point  $j(\lambda)=256\frac{(\lambda^2-\lambda+1)^3}{(\lambda^2-\lambda)^2}$ coincides with $-32768$. 
The only value of $\tau$ left in the list \eqref{list1} is $\frac{1+\sqrt{11}{\rm{i}}}{2}$. 
It is well-known that $j\left(\frac{1+\sqrt{11}{\rm{i}}}{2}\right)$ equals $-32768=-2^{15}$ 
(for instance, check the table on page 383 of \cite{Cohen}). This fact is moreover confirmed by $q$-expansion:
$$ j\left(\frac{1+\sqrt{11}{\rm{i}}}{2}\right)\approx-{\rm{e}}^{\pi\sqrt{11}}+744-196884{\rm{e}}^{-\pi\sqrt{11}}+21493760{\rm{e}}^{-2\pi\sqrt{11}}\approx  -32767.999977$$
\end{example}

\begin{example}\label{D=4}
Next, consider the case of $D=4$. The quadratic imaginary number fields in the form of $\Bbb{Q}\left(\sqrt{a^2-4\times 4}\right)$ occurring in \ref{classification} 
when $D=4$ are $\Bbb{Q}\left(\sqrt{-1}\right),\Bbb{Q}\left(\sqrt{-15}\right),\Bbb{Q}\left(\sqrt{-3}\right)$ and $\Bbb{Q}\left(\sqrt{-7}\right)$ where $a=0,1,2,3$ respectively. 
All of them satisfy the unique factorization property except $\Bbb{Q}\left(\sqrt{-15}\right)$ whose class number is two. 
We shall determine a set of representatives for orbits of points $\tau=\frac{1}{2b}\left(u+\sqrt{16-a^2}{\rm{i}}\right)\,(4b\mid u^2+16-a^2)$ under the action of ${\rm{SL}_2(\Bbb{Z})}$. 
For $a=0,2$, fields $\Bbb{Q}\left(\sqrt{-1}\right),\Bbb{Q}\left(\rm{\sqrt{-3}}\right)$ both have class number one and the corresponding quadratic binary forms are of discriminant $a^2-4D=-16\,\text{or}\,-12$. 
In each case, associated binary quadratic forms $bx^2+uxy+\frac{u^2+16-a^2}{4b}y^2$ are equivalent under the action of ${\rm{SL}_2(\Bbb{Z})}$ provided the ${\rm{GCD}}$ of coefficients is fixed. 
Conditioning on ${\rm{GCD}}$, when $a=0$ there are quadratic forms of discriminant $-16$: $x^2+4y^2\,(u=0,b=1)$ associated with $\tau=2{\rm{i}}$ and $2(x^2+y^2)\,(u=0,b=2)$ 
associated with $\tau={\rm{i}}$ and the case when all coefficients of $bx^2+uxy+\frac{u^2+16}{4b}y^2$ are divisible by four is impossible. For $a=2$, again the class number is one. 
We are dealing with binary quadratic forms of discriminant $-12$. Hence their orbit under the action of ${\rm{SL}_2(\Bbb{Z})}$ is represented by $x^2+3y^2\,(u=0,b=1)$ for primitive forms along with $2(x^2+xy+y^2)\,(u=b=2)$ that are associated with $\tau=\sqrt{3}{\rm{i}}$ 
and $\tau=\frac{1+\sqrt{3}{\rm{i}}}{2}$. When $a=3$, class number of $\Bbb{Q}\left(\sqrt{-7}\right)$ is one and discriminant of the corresponding quadratic forms are $a^2-4D=-7$ which is square free. 
Thus all of them are primitive and there is only one orbit, for example that of $\tau=\frac{1+\sqrt{7}{\rm{i}}}{2}\,(u=1,b=1)$. Finally, suppose $a=1$. 
Class number of $\Bbb{Q}\left(\sqrt{-15}\right)$ is two and the associated forms are of discriminant $a^2-4D=-15$ and therefore primitive. 
Hence there are at most two orbits. 
Without much trouble, it can be verified that binary quadratic forms of discriminant $-15$, $x^2+xy+4y^2\,(u=b=1)$ associated with $\frac{1+\sqrt{15}{\rm{i}}}{2}$ and $2x^2+xy+2y^2\,(u=1,b=2)$ associated with $\tau=\frac{1+\sqrt{15}{\rm{i}}}{4}$
are not equivalent under the action of ${\rm{SL}_2(\Bbb{Z})}$.  This fact completes our set of representatives as:  
\begin{equation}\label{list2}
\rm{i},2\rm{i},\frac{1+\sqrt{3}{\rm{i}}}{2},\frac{1+\sqrt{7}{\rm{i}}}{2},\frac{1+\sqrt{15}{\rm{i}}}{2},\frac{1+\sqrt{15}{\rm{i}}}{4}
\end{equation}
When $D=4$,  \eqref{system5} and \eqref{system6} may be written as 
\begin{equation}\label{system7}
\begin{cases}
k(x-1)(x-\lambda)(x-\alpha)^2-x(x-\beta)^2=k(x^2+ux+v)^2\\
k(x-1)(x-\lambda)(x-\alpha)^2-\lambda x(x-\beta)^2=k(x^2+u^{\prime}x+v^{\prime})^2
\end{cases}
\end{equation}
\begin{equation}\label{system8}
\begin{cases}
kx(x-1)(x-\alpha)^2-(x-\lambda)(x-\beta)^2=k(x^2+ux+v)^2\\
kx(x-1)(x-\alpha)^2-\lambda(x-\lambda)(x-\beta)^2=k(x^2+u^{\prime}x+v^{\prime})^2
\end{cases}
\end{equation}
Equality of coefficients in them results in  systems of eight polynomial equations with eight unknowns including $\lambda$. 
According to Remark \ref{perfectsquare}, evaluating $j(\lambda)$ at any solution provides us with a value of the modular function $j$ at one of points from the upper half plane listed in \eqref{list2} though $D=4$ is a perfect square. The system \eqref{system7} is much easier to solve and yields:
\begin{equation*}
\begin{split}
&\begin{cases}
-\frac{1}{4}(x-1)(x-\beta^2)(x+\beta)^2-x(x-\beta)^2=-\frac{1}{4}\big(x^2+(4\beta+2)x+\beta^2\big)^2\\
-\frac{1}{4}(x-1)(x-\beta^2)(x+\beta)^2-\beta^2x(x-\beta)^2=-\frac{1}{4}\big(x^2-(8\beta+2)x+\beta^2\big)^2
\end{cases}\\
&\qquad\qquad\qquad\beta\in\{-3\pm 2\sqrt{2}\}\quad\lambda=\beta^2\in\{17\pm 12\sqrt{2}\}
\end{split}
\end{equation*}
and also:
\begin{equation*}
\begin{split}
&\begin{cases}
\frac{\beta^2-1}{4(u+2\beta)}(x-1)(x-\beta^2)(x+\beta)^2-x(x-\beta)^2=\frac{\beta^2-1}{4(u+2\beta)}\big(x^2+ux+\beta^2\big)^2\\
\frac{\beta^2-1}{4(u+2\beta)}(x-1)(x-\beta^2)(x+\beta)^2-\beta^2x(x-\beta)^2=\frac{\beta^2-1}{4(u+2\beta)}\big(x^2-(u+4\beta)x+\beta^2\big)^2
\end{cases}\\
&\qquad\beta\in\{\pm{\rm{i}}(2\pm \sqrt{3})\}\quad
\lambda=\beta^2\in\{-7\pm 4\sqrt{3}\}\quad
u=\beta+1-\frac{(\beta+\frac{1}{\beta})(\beta-1)}{2}
\end{split}
\end{equation*}
In first solution, $j(\lambda)=256\frac{(\lambda^2-\lambda+1)^3}{(\lambda^2-\lambda)^2}\mid_{17\pm 12\sqrt{2}}=66^3=287496$ while for $\tau=2{\rm{i}}$ in \eqref{list2}:
$$j(2i)\approx{\rm{e}}^{4\pi}+744+196884{\rm{e}}^{-4\pi}+21493760{\rm{e}}^{-8\pi}≈\approx287495.999999$$
In the second solution to system \eqref{system7} presented above, $\lambda=-7\pm 4\sqrt{3}$ 
which corresponds to $j\left(\sqrt{3}{\rm{i}}\right)=16\times 15^3$, 
just like what we saw before in Example  \ref{D=3}. 
In the case of system \eqref{system8}, 
after some rather cumbersome algebraic manipulations one obtains the following identities as all the solutions to \eqref{system8}: 
\begin{equation*}
\begin{split}
&\begin{cases}
&\frac{t(t^2+1)}{2}x(x-1)\big(x-\frac{(t+1)^2}{2(t^2+1)}\big)^2-(x-t^4)\big(x-\frac{(t+1)^2}{4t}\big)^2=\frac{t(t^2+1)}{2}\big(x^2-(1+\frac{1}{t})x+t^3\big)^2\\
&\frac{t(t^2+1)}{2}x(x-1)\big(x-\frac{(t+1)^2}{2(t^2+1)}\big)^2-t^4(x-t^4)\big(x-\frac{(t+1)^2}{4t}\big)^2=\frac{t(t^2+1)}{2}\big(x^2-(t+1)x+t^5\big)^2
\end{cases}\\
&t\in\left\{\frac{-1\pm\sqrt{7}{\rm{i}}}{4},\frac{\sqrt{5}-1}{8}\left(1\pm\left(2\sqrt{3}+
\sqrt{15}\right){\rm{i}}\right),\frac{\sqrt{5}+1}{8}\left(-1\pm\left(2\sqrt{3}-\sqrt{15}\right){\rm{i}}\right)\right\}\quad \lambda=t^4
\end{split}
\end{equation*}
Then we should evaluate $j(\lambda)=256\frac{(\lambda^2-\lambda+1)^3}{(\lambda^2-\lambda)^2}$ at $\lambda=t^4$ for aforementioned values of $t$. 
When $t=\frac{-1\pm\sqrt{7}{\rm{i}}}{4}$, we have $\lambda=t^4=\frac{1\pm 3\sqrt{7}{\rm{i}}}{32}$ at which $j$-invariant is $-15^3$, 
the value that was previously established as $j\left(\frac{1+\sqrt{7}{\rm{i}}}{2}\right)$ in Example \ref{D=2}. 
For other four values of $t$
\footnote{We should explain how these calculations were carried out. Numbers $t\neq 1$ appeared here are actually among the roots of $8t^5(t^2+1)=(t+1)^4$. 
Factorizing  the polynomial $8t^5(t^2+1)-(t+1)^4$ shows that either $\frac{t^2+t}{t^2+1}=-1$ where $t=\frac{-1+\sqrt{7}{\rm{i}}}{4}$ or  $\frac{t^2+t}{t^2+1}=2+\sqrt{5}$ 
where $t=\frac{\sqrt{5}-1}{8}\left(1\pm \left(2\sqrt{3}+\sqrt{15}\right){\rm{i}}\right)$ or $\frac{t^2+t}{t^2+1}=2-\sqrt{5}$ 
where $t=\frac{\sqrt{5}+1}{8}\left(-1\pm\left(2\sqrt{3}-\sqrt{15}\right){\rm{i}}\right)$. Using them, a quadratic equation which $\lambda =t^4$ satisfies is determined. 
Then a simple observation was employed stating that the value of  $j(\lambda)=256\frac{(\lambda^2-\lambda+1)^3}{(\lambda^2-\lambda)^2}$ 
at roots of $\lambda^2-\mu\lambda+\mu=0$ equals $256\frac{(\mu-1)^3}{\mu}$.}:
$$\begin{cases}
j(\lambda)\Big |_{\left(\frac{\sqrt{5}-1}{8}\left(1\pm\left(2\sqrt{3}+\sqrt{15}\right){\rm{i}}\right)\right)^4}=\frac{(7-3\sqrt{5})^2(-1+3\sqrt{5})^3(15+3\sqrt{5})^3}{256}\\
j(\lambda))\Big |_{\left(\frac{\sqrt{5}+1}{8}\left(-1\pm\left(2\sqrt{3}-\sqrt{15}\right){\rm{i}}\right)\right)^4}=\frac{(7+3\sqrt{5})^2(-1-3\sqrt{5})^3(15-3\sqrt{5})^3}{256}
\end{cases}$$
We will determine to which points of the upper half plane exhibited in \eqref{list2} these values correspond. First, using some numerical software:
$$\frac{(7-3\sqrt{5})^2 (-1+3\sqrt{5})^3 (15+3\sqrt{5})^3}{256}\approx632.832862$$
Let $\tau$ be $\frac{1+\sqrt{15}{\rm{i}}}{4}$ from our list. Then $j(\tau)$ should be real because $\mid\tau\mid=1$. 
On the other hand, $q=e^{2\pi{\rm{i}}\tau}$ is purely imaginary. 
So in the $q$-expansion only even powers of $q$ matters:
\begin{equation*}
\begin{split}
&j\left(\frac{1+\sqrt{15}{\rm{i}}}{4}\right)\approx744+21493760q^2+20245856256q^4\\
&=744-21493760{\rm{e}}^{-\pi\sqrt{15}}+20245856256e^{-2\pi\sqrt{15}}\approx632.833459
\end{split}
\end{equation*}
Therefore: $j\left(\frac{1+\sqrt{15}{\rm{i}}}{4}\right)=\frac{(7-3\sqrt{5})^2 (-1+3\sqrt{5})^3 (15+3\sqrt{5})^3}{256}$. 
Secondly, we have the estimate:
$$\frac{(7+3\sqrt{5})^2 (-1-3\sqrt{5})^3 (15-3\sqrt{5})^3}{256}\approx-191657.832862$$
while $q$-expansion gives:
$$ j\left(\frac{1+\sqrt{15}{\rm{i}}}{2}\right)\approx-{\rm{e}}^{\pi\sqrt{15}}+744-196884{\rm{e}}^{-\pi\sqrt{15}}+21493760{\rm{e}}^{-2\pi\sqrt{15}}\approx-191657.832862$$
which coincides with previous estimate to six digits after decimal point. 
Hence $j\left(\frac{1+\sqrt{15}{\rm{i}}}{2}\right)=\frac{(7+3\sqrt{5})^2 (-1-3\sqrt{5})^3 (15-3\sqrt{5})^3}{256}$.  
Both values of $j$-invariant we just derived belong to $\Bbb{Q}(\sqrt{5})\setminus \Bbb{Q}$. 
Thus we obtain the Hilbert class field of $\Bbb{Q}\left(\sqrt{-15}\right)=\Bbb{Q}\left(\frac{1+\sqrt{15}{\rm{i}}}{2}\right)$ as $\Bbb{Q}\left(\sqrt{-15},j\left(\frac{1+\sqrt{15}{\rm{i}}}{2}\right)\right)=\Bbb{Q}\left(\sqrt{-3},\sqrt{5}\right)$.
\end{example}

\section{Recovering $j(2\tau)$ and $j(3\tau)$ form $j(\tau)$}
Our main idea, i.e. reducing $j$-invariant calculations to studying certain classes of meromorphic functions on the Riemann sphere and eventually solving systems of polynomial equations, can also be used for computing $j$ values for degree $D$ isogenous elliptic curves $f:\mathcal{E}^\prime\rightarrow\mathcal{E}$. Again, $f$ induces a meromorphic function of same degree on $\Bbb{CP}^1$ fitting in a commutative diagram similar to $(\star)$.

%\begin{figure}[ht]
%\centering
%\includegraphics{diagram4.png}
%\end{figure}

\begin{equation*} \tag{$\star\star\star$}
\xymatrix{\mathcal{E}^{\prime}  \ar[r]^f \ar[d]  & \mathcal{E} \ar[d]\\
                 \Bbb{CP}^1 \ar[r]^h            &  \Bbb{CP}^1 }
\end{equation*}

Fixing $D$ and with prior knowledge of the Legendre form $y^2=x(x-1)(x-\lambda)$ for $\mathcal{E}$ 
and a normalized period $\tau\in\Bbb{H}$ of $\mathcal{E}$, 
the goal is to compute the $j$-invariant of $\mathcal{E}^{\prime}$. The   Legendre form of  $\mathcal{E}^{\prime}$ is 
assumed to be $y^2=x(x-1)(x-\lambda^{\prime})$. 
Normalized periods $\tau^{\prime}$ of  $\mathcal{E}^{\prime}$ are exactly images of $\tau$ under integral $2\times 2$ matrices of determinant $D$ in the action of ${\rm{GL}}_{2}^+(\Bbb{R})$ on 
the upper half plane. 
This set, i.e. $\big \{\frac{a\tau+b}{c\tau+d}\mid a,b,c,d\in\Bbb{Z}, ad-bc=D\big\}$, decomposes as union of finitely many ${\rm{SL}}_2(\Bbb{Z})$-orbits which once more provides us with a finite list of points (denoted by $\tau^{\prime}$) in the upper half plane whose corresponding Legendre form representation 
(that is, $\lambda^{\prime}$) is going to be investigated by studying certain degree $D$ meromorphic function $h:\Bbb{CP}^1\rightarrow\Bbb{CP}^1$. 
The commutative diagram $(\star\star\star)$ implies that $h$ obeys constraints similar to those in 
Proposition \ref{constraint} with a little modification that 
$\lambda$ in the domain should be replaced with $\lambda^{\prime}$:
\begin{itemize}
\item $h(\infty)=\infty$
\item $h$ has $2D-2$ critical points each with multiplicity $2$.
\item Critical values of $h$ belong to the set $\{0,1,\lambda,\infty\}$.
\item $h$ has simple points $0,1,\lambda^\prime,\infty$ whose value at any of them lies in $\{0,1,\lambda,\infty\}$.
\end{itemize}
The plan is to start with $D,\tau,\lambda$, form a finite list of ${\rm{SL}}_2(\Bbb{Z})$-orbits of $\tau^{\prime}$'s by exhibiting representatives, determine available $\lambda^{\prime}$'s -at which value of $j:\Bbb{C}-\{0,1\}\rightarrow\Bbb{C}$ may easily be calculated- by studying aforementioned functions $h$ and finally using $q$-expansion to find out a $\tau^{\prime}$ from our first list corresponds to which of the $j$-values just obtained.   

Let us apply this to $D=2$. Any integer $2\times 2$ matrix of determinant two may be written as an element of ${\rm{SL}}_2(\Bbb{Z})$ multiplied from right by one of matrices 
$\begin{bmatrix}
2&0\\
0&1
\end{bmatrix}$,
$\begin{bmatrix}
1&0\\
0&2
\end{bmatrix}$
or
$\begin{bmatrix}
1&1\\
-1&1
\end{bmatrix}$.
Consequently, having $\tau\in\Bbb{H}$ and a corresponding $\lambda\in\Bbb{C}-\{0,1\}$ in hand, we are going to derive values of $j:\Bbb{H}\rightarrow\Bbb{C}$ at points $2\tau,\frac{\tau}{2},\frac{\tau+1}{-\tau+1}$ from its value at $\tau$. 
The degree two map $h:\Bbb{CP}^1\rightarrow\Bbb{CP}^1$ has two critical points each of multiplicity two. 
So from conditions mentioned above, its critical values are two elements in $\{0,1,\lambda\}$ 
and furthermore the fiber over the remaining point consists of two points in $\{0,1,\lambda^{\prime}\}$. The point left in this set is mapped to $\infty$. 
It must be mentioned that these conditions on $h(x)$ are also sufficient in the sense that for 
any such a $h(x)$ ($\lambda^{\prime}$ its unique critical point outside of $\{0,1,\infty\}$) 
fits in bottom of a diagram like $(\star\star\star)$ 
whose top row is a degree two homomorphism $f$ from  
$\mathcal{E}^{\prime}=\left\{y^2=x(x-1)(x-\lambda^{\prime})\right\}$ to 
$\mathcal{E}=\left\{y^2=x(x-1)(x-\lambda)\right\}$ and left and right columns are the ramified 2-fold covers $(x,y)\mapsto x$. 
This is due to the fact that these conditions force $\frac{h(x)(h(x)-1)(h(x)-\lambda)}{x(x-1)(x-\lambda^{\prime})}\in\Bbb{C}(x)$ to have a square root, say $g(x)$. 
Then $f:(x,y)\mapsto\big(h(x),yg(x)\big)$ is a well-defined map making 
the diagram commutative.  
Substituting $\lambda$ and $\lambda^{\prime}$ with other points in their ${\rm{S}_3}$-orbit via combining  $h$ from left with one of M\"obius transformations 
$x\mapsto 1-x\,\text{or}\,1-\frac{x}{\lambda}$ and from right with one of M\"obius transformations $x\mapsto{\lambda^{\prime}}x\,\text{or}\,\lambda^{\prime}(1-x)$,
without any loss of generality we may concentrate only on the case  where under $h$: $\lambda^{\prime},\infty\mapsto \infty$, $0,1\mapsto 0$ and $1,\lambda$ are critical values. 
Hence  $h(x)=\frac{kx(x-1)}{x-\lambda^{\prime}}$ and since $1,\lambda$ are critical values, discriminant of quadratic polynomials $kx(x-1)-(x-\lambda^{\prime})$ and $kx(x-1)-\lambda(x-\lambda^{\prime})$ must vanish. 
This leads to a very simple system of equations with unknowns $k,\lambda^{\prime}$ whose solutions are $k=\sqrt{\lambda},\lambda^{\prime}=\frac{1}{4}\left(\sqrt{\lambda}+\frac{1}{\sqrt{\lambda}}+2\right)$. 
Thus $j$-invariant of $\mathcal{E}^{\prime}$ is in the form of:
$$256\frac{\left[\left(\frac{1}{4}\left(\sqrt{\lambda}+\frac{1}{\sqrt{\lambda}}+2\right)\right)^2-\frac{1}{4}\left(\sqrt{\lambda}+\frac{1}{\sqrt{\lambda}}+2\right)+1\right]^3}{\left[\left(\frac{1}{4}\left(\sqrt{\lambda}+\frac{1}{\sqrt{\lambda}}+2\right)\right)^2-\frac{1}{4}\left(\sqrt{\lambda}+\frac{1}{\sqrt{\lambda}}+2\right)\right]^2}=16\frac{\left(\lambda+\frac{1}{\lambda}+14\right)^3}{\left(\lambda+\frac{1}{\lambda}-2\right)^2} $$
where $\lambda$ varies in its ${\rm{S}}_3$-orbit $\left\{\lambda,1-\lambda,\frac{1}{\lambda},\frac{1}{1-\lambda},\frac{\lambda-1}{\lambda},\frac{\lambda}{\lambda-1}\right\}$. 
As a result, the following Theorem is inferred:

\begin{thm} \label{2tau}
Let values of the modular function $j:\Bbb{H}\rightarrow\Bbb{C}$ and the function
$\begin{cases}
j:\Bbb{C}-\{0,1\}\rightarrow\Bbb{C}\\
\lambda\mapsto 256\frac{(\lambda^2-\lambda+1)^3}{(\lambda^2-\lambda)^2}
\end{cases}$ 
coincide at points $\tau$ and $\lambda$ respectively. Then:
$$\left\{j\left(2\tau\right),j\left(\frac{\tau}{2}\right),j\left(\frac{\tau+1}{-\tau+1}\right)\right\}=\left\{16\frac{\left(u+\frac{1}{u}+14\right)^3}{\left(u+\frac{1}{u}-2\right)^2}\Bigg | u\in\left\{\lambda,1-\lambda,\frac{\lambda-1}{\lambda}\right\}\right\} $$
\end{thm}

\begin{example}\label{ex2tau}
To illustrate power of this Theorem, consider some 
$j$-invariant values derived in previous examples. 
For $\tau=\sqrt{2}{\rm{i}},\lambda=3+2\sqrt{2}$ is a corresponding $\lambda$ according to Examples \ref{D=2} and \ref{D=3}. 
Evaluating  the expression $16\frac{\big(u+\frac{1}{u}+14\big)^3}{\big(u+\frac{1}{u}-2\big)^2}$ at $u=\frac{\lambda-1}{\lambda}=2\sqrt{2}-2$ 
yields \hspace{1.5cm}$10^3\left(5+\sqrt{2}\right)^{3}\left(7+5\sqrt{2}\right)^2\approx 52249767.137718$ while $q$-expansion approximates $j\left(2\sqrt{2}{\rm{i}}\right)$ as
$$j\left(2\sqrt{2}{\rm{i}}\right)\approx {\rm{e}}^{4\pi\sqrt{2}}+744+196884{\rm{e}}^{-4\pi\sqrt{2}}\approx 52249767.137718$$
that coincides with $10^3\left(5+\sqrt{2}\right)^{3}\left(7+5\sqrt{2}\right)^2$ to six digits after decimal point. 
We deduce that $j\left(2\sqrt{2}{\rm{i}}\right)=10^3\left(5+\sqrt{2}\right)^{3}\left(7+5\sqrt{2}\right)^2$ and moreover $\tau^{\prime}=2\sqrt{2}{\rm{i}}$ corresponds to 
$$\lambda^{\prime}=\frac{1}{4}\left(\sqrt{\frac{\lambda-1}{\lambda}}+\sqrt{\frac{\lambda}{\lambda-1}}+2\right)\Big |_{\lambda=3+2\sqrt{2}}=\frac{1}{4}\left(\sqrt{2\sqrt{2}-2}+\frac{\sqrt{2\sqrt{2}+2}}{2}+2\right)$$
 and this method can be iterated another time to find $j\left(4\sqrt{2}{\rm{i}}\right)$ and so on. 
Note that here one other choices for $u$ in Theorem \ref{2tau} is $\lambda=3+2\sqrt{2}$ 
where  $16\frac{\big(u+\frac{1}{u}+14\big)^3}{\big(u+\frac{1}{u}-2\big)^2}$  
equals $20^3=j\left(\sqrt{2}{\rm{i}}\right)=j\left(\frac{\sqrt{2}{\rm{i}}}{2}\right)=j\left(\frac{\tau}{2}\right)$. 
So the only value of  $16\frac{\big(u+\frac{1}{u}+14\big)^3}{\big(u+\frac{1}{u}-2\big)^2}$ left, 
i.e. its value at $u=1-\lambda=-2\sqrt{2}-2$ can be the 
value of the modular function at $\frac{\tau+1}{-\tau+1}=\frac{-1+2\sqrt{2}{\rm{i}}}{3}$. 
Hence: $j\left(\frac{-1+2\sqrt{2}{\rm{i}}}{3}\right)= 10^3\left(5-\sqrt{2}\right)^3\left(7-5\sqrt{2}\right)^2$. 
Before finishing this example, we will quickly mention few other results of this type 
where we calculate the value of the $j$-function at point $2\tau$ of the 
upper half plane using values $j(\tau)$ currently available to us from previous examples. 
Consider pairs $(\tau,\lambda)=\left(\frac{1+\sqrt{7}{\rm{i}}}{2},\frac{1+ 3\sqrt{7}{\rm{i}}}{2}\right),\left(\sqrt{3}{\rm{i}},-7- 4\sqrt{3}\right),\left(2{\rm{i}},17+12\sqrt{2}\right)$  that
appeared respectively in Examples \ref{D=2}, \ref{D=3} and \ref{D=4}. 
Setting $u=\frac{\lambda-1}{\lambda}$, and computing $16\frac{\big(u+\frac{1}{u}+14\big)^3}{\big(u+\frac{1}{u}-2\big)^2}$ in these cases, we have:
\begin{equation*}
\begin{cases}
16\frac{\big(u+\frac{1}{u}+14\big)^3}{\big(u+\frac{1}{u}-2\big)^2}\Big |_{\frac{31+ 3\sqrt{7}{\rm{i}}}{32}}=255^3=16581375\\
16\frac{\big(u+\frac{1}{u}+14\big)^3}{\big(u+\frac{1}{u}-2\big)^2})\Big |_{8-4\sqrt{3}}=4\times 15^3\left(6-\sqrt{3}\right)^3\left(26+15\sqrt{3}\right)^2\approx 2835807690.422285\\
16\frac{\big(u+\frac{1}{u}+14\big)^3}{\big(u+\frac{1}{u}-2\big)^2})\Big |_{12\sqrt{2}-16}=\frac{1}{2}\left(99\sqrt{2}-12\right)^3\left(99\sqrt{2}+140\right)^2\approx 82226316329.59503
\end{cases}
\end{equation*}
which by Theorem \ref{2tau} are candidates for $j\left(1+\sqrt{7}{\rm{i}}\right)=j\left(\sqrt{7}{\rm{i}}\right),j\left(2\sqrt{3}{\rm{i}}\right),j\left(4{\rm{i}}\right)$ in same order. 
The $q$-expansion assures us that they are correct choices:
\begin{equation*}
\begin{split}
&j\left(\sqrt{7}{\rm{i}}\right)\approx {\rm{e}}^{2{\rm{\pi}}\sqrt{7}}+744+196884{\rm{e}}^{-2{\rm{\pi}}\sqrt{7}}\approx 16581374.999999\\
&j\left(2\sqrt{3}{\rm{i}}\right)\approx {\rm{e}}^{4{\rm{\pi}}\sqrt{3}}+744+196884{\rm{e}}^{-4{\rm{\pi}}\sqrt{3}}\approx 2835807690.422278\\
&j(4{\rm{i}})\approx {\rm{e}}^{8{\rm{\pi}}}+744+196884{\rm{e}}^{-8{\rm{\pi}}}\approx 82226316329.59491
\end{split}
\end{equation*}
\end{example}

The same procedure is applicable when the degree $D$ of rows in $(\star\star\star)$ is three although 
the calculations are more tedious. 
First, note that any $2\times 2$ integer matrix of determinant three has a description as a product of an element of ${\rm{SL}}_2(\Bbb{Z})$ by one of matrices 
$\begin{bmatrix}
3&0\\
0&1
\end{bmatrix}$,
$\begin{bmatrix}
1&0\\
0&3
\end{bmatrix}$,
$\begin{bmatrix}
1&2\\
-1&1
\end{bmatrix}$
or
$\begin{bmatrix}
1&-2\\
1&1
\end{bmatrix}$.
So we hope to recover some numbers among $j(3\tau),j\left(\frac{\tau}{3}\right),j\left(\frac{\tau+2}{-\tau+1}\right),j\left(\frac{\tau-2}{\tau+1}\right)$ 
with the assumption that $\tau$ is specified along with  $j(\tau)$   in the form of $j(\lambda)=256\frac{(\lambda^2-\lambda+1)^3}{(\lambda^2-\lambda)^2}$. 
Secondly, we are going to accomplish this via studying certain types of  degree three meromorphic functions. These meromorphic functions $h:\Bbb{CP}^1\rightarrow\Bbb{CP}^1$ which we are interested in, 
by same argument as the one resulted in Corollary \ref{oddcase}, 
should satisfy conditions similar to those in Corollary \ref{oddcase} for $D=3$ 
with the exception that $\lambda$ in the domain must be replaced with $\lambda^{\prime}$. 
Again, by composing with suitable M\"obius maps, without changing the ${\rm{S}}_3$-orbits of $\lambda$ or $\lambda^{\prime}$, 
we may assume that the first case in Corollary \ref{oddcase} occurs.  
Therefore critical values of $h$ are exactly $0,1,\lambda,\infty$, over each $h$ has one critical point of multiplicity two  
and one simple point which are $0,1,\lambda^{\prime},\infty$ respectively. 
The constraints just mentioned are also sufficient, that is, any such a function $h$ results in a degree three homomorphism of elliptic curves making $(\star\star\star)$ commutative. 
Hence in order to get the value of $j$-invariant at one of the aforementioned four points of the
upper half plane associated with $\tau$, we will try to exhibit such a function in terms of $\lambda$. 
Since $h(0)=0$ and $h(\infty)=\infty$ and $0,\infty$ are critical values, 
$h(x)=\frac{xP(x)}{Q(x)}$ where $P(x)$ and $Q(x)$ are quadratic polynomials of discriminant zero. $h(1)=1,h(\lambda^{\prime})=\lambda$ yield $Q(1)=P(1), Q(\lambda^{\prime})=\frac{\lambda^{\prime}}{\lambda}P(\lambda^{\prime})$. 
So writing them in the basis $\left\{(x-1)^2,(x-1)(x-\lambda^{\prime}),(x-\lambda^{\prime})^2\right\}$ for the space of polynomials of degree less than three, 
we conclude that for some suitable $t\in\Bbb{C}-\{0,-1\}$:
\begin{equation} \label{prototype}
h(x)=\frac{x\big((x-1)^2+2t(x-1)(x-\lambda^{\prime})+t^2(x-\lambda^{\prime})^2\big)}{\frac{\lambda^{\prime}}{\lambda}(x-1)^2+2t\sqrt{\frac{\lambda^{\prime}}{\lambda}}(x-1)(x-\lambda^{\prime})+t^2(x-\lambda^{\prime})^2} 
\end{equation}
The only condition left is that $1,\lambda$ are critical values of $h(x)$, i.e. numerators of $h(x)-1$ and $h(x)-\lambda$ have multiple roots.  The polynomials that may be written as:
$$\begin{cases}
(x-1)\left[(x-1)(x-\frac{\lambda^{\prime}}{\lambda})+2t\left(x-\sqrt{\frac{\lambda^{\prime}}{\lambda}}\right)(x-\lambda^{\prime})+t^2(x-\lambda^{\prime})^2)\right]\\
(x-\lambda^{\prime})\left[(x-1)^2+2t\left(x-\sqrt{\lambda\lambda^{\prime}}\right)(x-1)+t^2(x-\lambda)(x-\lambda^{\prime})\right]
\end{cases}$$
The vanishing of discriminants of quadratic polynomials that appeared in brackets yields a system with unknowns $t,\lambda^{\prime}$:
$$\begin{cases}
\left(1+\frac{\lambda^{\prime}}{\lambda}+2t\left(\sqrt{\frac{\lambda^{\prime}}{\lambda}}+\lambda^{\prime}\right)+2t^2\lambda^{\prime}\right)^2=4(t+1)^2\left(\sqrt\frac{\lambda^{\prime}}{\lambda}+t\lambda^{\prime}\right)^2\\
\left(2+2t\left(\sqrt{\lambda\lambda^{\prime}}+1\right)+t^2(\lambda+\lambda^{\prime})\right)^2=4(t+1)^2\left(1+t\sqrt{\lambda\lambda^{\prime}}\right)^2
\end{cases}$$
We are looking for solutions in terms of $\lambda$. 
Assuming $\lambda\neq\lambda^{\prime}$, second equation gives us the identity $4+4t\left(\sqrt{\lambda\lambda^{\prime}}+1\right)+t^2(\sqrt{\lambda}+\sqrt{\lambda^{\prime}})^2=0$ 
while with dividing first equation by second one, we have:
$$\frac{1+\frac{\lambda^{\prime}}{\lambda}+2t\left(\sqrt{\frac{\lambda^{\prime}}{\lambda}}+\lambda^{\prime}\right)+2t^2\lambda^{\prime}}{2+2t\left(\sqrt{\lambda\lambda^{\prime}}+1\right)+t^2(\lambda+\lambda^{\prime})}=\pm\frac{\sqrt{\lambda^{\prime}}}{\sqrt{\lambda}}.$$
Choosing the plus sign, this identity reduces to $t^2=\frac{1}{\sqrt{\lambda\lambda^{\prime}}}$. 
Combining it with  $4+4t\left(\sqrt{\lambda\lambda^{\prime}}+1\right)+t^2\left(\sqrt{\lambda}+\sqrt{\lambda^{\prime}}\right)^2=0$ appeared before gives us:
$$\left(\lambda+\lambda^{\prime}+6\sqrt{\lambda\lambda^{\prime}}\right)^2=16\left(1+\sqrt{\lambda\lambda^{\prime}}\right)^2\sqrt{\lambda\lambda^{\prime}}$$
Hence fixing $\lambda$, for any $\lambda^{\prime}$ satisfying the above equation and $t$ given by $t^2=\frac{1}{\sqrt{\lambda\lambda^{\prime}}}$, $h(x)$ in \ref{prototype} will be a  degree three meromorphic function 
on $\Bbb{CP}^1$ with desired properties.  
This is reflected in our last Theorem: 
\begin{thm}\label{3tau}
Let $\tau\in\Bbb{H}$ and $\lambda\in\Bbb{C}-\{0,1\}$ be such that $j(\tau)=256\frac{(\lambda^2-\lambda+1)^3}{(\lambda^2-\lambda)^2}$. If  $\lambda^{\prime}\in\Bbb{C}-\{0,1\}$  satisfies $\left(\lambda+\lambda^{\prime}+6\sqrt{\lambda\lambda^{\prime}}\right)^2=16\left(1+\sqrt{\lambda\lambda^{\prime}}\right)^2\sqrt{\lambda\lambda^{\prime}}$, then  $256\frac{({\lambda^{\prime}}^2-\lambda^{\prime}+1)^3}{({\lambda^{\prime}}^2-\lambda^{\prime})^2}$ belongs to the set $\left\{j(3\tau),j\left(\frac{\tau}{3}\right),j\left(\frac{\tau+2}{-\tau+1}\right),j\left(\frac{\tau-2}{\tau+1}\right)\right\}$.
\end{thm} 

\begin{example}\label{ex3tau}
We finish with an application of the Theorem just proved. Suppose $\lambda=-1$ which corresponds to the square lattice, i.e. $\tau={\rm{i}}$, and the elliptic curve $y^2=x^3-x$. Employing Theorem \ref{3tau},
the objective is to determine some of numbers $j(3{\rm{i}})=j\left(\frac{{\rm{i}}}{3}\right)$ and $j\left(\frac{\pm 1+3{\rm{i}}}{2}\right)$. 
We have following estimates for them:
\begin{equation*}
\begin{split}
&j(3{\rm{i}})\approx{\rm{e}}^{6\pi}+744+196884{\rm{e}}^{-6\pi}\approx 153553679.396728\\
&j\left(\frac{\pm1+3{\rm{i}}}{2}\right)\approx -{\rm{e}}^{3\pi}+744-196884{\rm{e}}^{-3\pi}+21493760{\rm{e}}^{-6\pi}\approx  -11663.396275 
\end{split}
\end{equation*}
 $\sqrt{\lambda^{\prime}}$ should be a root of 
$$(x^2+6{\rm{i}}x-1)^2-16{\rm{i}}x(1+{\rm{i}}x)^2 =x^4+28{\rm{i}}x^3-6x^2-28{\rm{i}}x+1$$
Finding roots with help of some computer software, it turns out that our choices for $\sqrt{\lambda^{\prime}}$ are $-{\rm{i}}\left(2+\sqrt{3}\right)\left(\sqrt{2}\pm\sqrt[4]{3}\right)^2$ and  $-{\rm{i}}\left(2-\sqrt{3}\right)\left(\sqrt{2}\pm{\rm{i}}\sqrt[4]{3}\right)^2$ that give us: $$\lambda^{\prime}=-\left(2+\sqrt{3}\right)^2\left(\sqrt{2}\pm\sqrt[4]{3}\right)^4\quad\text{or}\quad\lambda^{\prime}=-\left(2-\sqrt{3}\right)^2\left(\sqrt{2}\pm{\rm{i}}\sqrt[4]{3}\right)^4.$$ 
In each of them the numbers that appeared are reciprocals of each other. 
So let us only concentrate on $\lambda^{\prime}=-\left(2+\sqrt{3}\right)^2\left(\sqrt{2}+\sqrt[4]{3}\right)^4$ and $\lambda^{\prime}=-\left(2-\sqrt{3}\right)^2\left(\sqrt{2}-{\rm{i}}\sqrt[4]{3}\right)^4$. They satisfy quadratic equations ${\lambda^{\prime}}^2+2\left(193+112\sqrt{3}\right)\lambda^{\prime}+1=0$ and ${\lambda^{\prime}}^2+2\left(193-112\sqrt{3}\right)\lambda^{\prime}+1=0$
respectively, the fact that extremely simplifies calculation of $j(\lambda)=256\frac{(\lambda^2-\lambda+1)^3}{(\lambda^2-\lambda)^2}$ at them (cf. footnote to Example \ref{D=3}):
\begin{equation*}
\begin{cases}
j\left(\lambda\right)\big |_{-\left(2+\sqrt{3}\right)^2\left(\sqrt{2}+\sqrt[4]{3}\right)^4}=64\left(387+224\sqrt{3}\right)^3\left(97-56\sqrt{3}\right)\approx  153553679.396790\\
j\left(\lambda\right)\big |_{-\left(2-\sqrt{3}\right)^2\left(\sqrt{2}+{\rm{i}}\sqrt[4]{3}\right)^4}=64\left(387-224\sqrt{3}\right)^3\left(97+56\sqrt{3}\right)\approx -11663.396728\\
\end{cases}
\end{equation*}
Comparing with approximations obtained from $q$-expansion before, we deduce that 
$$j(3{\rm{i}})=64\left(387+224\sqrt{3}\right)^3\left(97-56\sqrt{3}\right)\quad j\left(\frac{\pm1+3{\rm{i}}}{2}\right)= 64\left(387-224\sqrt{3}\right)^3\left(97+56\sqrt{3}\right)$$
\end{example}

\section{Table}
Results obtained in the examples of this paper are summarized in the  Table \ref{tab:1} where for a period $\tau$  from the upper half plane appeared in the first column, the Legendre representation and the $j$-invariant of the elliptic curve $\mathcal{E}\cong\frac{\Bbb{C}}{\Bbb{Z}+\Bbb{Z}\tau}$ can be read off from the second and third columns while in the fourth column the imaginary quadratic order ${\rm{End}}(\mathcal{E})$ is shown.
\begin{center}
\begin{table}[hbt!]
\caption{The special values of the $j$-function computed in this paper}
\begin{tabular}{|c|c|c|c|}
\hline
$\tau\in\Bbb{H}$ & $\lambda\in\Bbb{C}-\{0,1\}$ & $j$-invariant & $R$\\
\hline
$\sqrt{2}{\rm{i}}$ & $3\pm 2\sqrt{2}$ & $20^3$ & $\Bbb{Z}[\sqrt{-2}]$\\
\hline
$\sqrt{3}{\rm{i}}$ & $-7\pm 4\sqrt{3}$ & $16\times 15^3$ & $\Bbb{Z}[\sqrt{-3}]$\\
\hline
$\frac{1+\sqrt{7}{\rm{i}}}{2}$ & $\frac{1\pm 3\sqrt{7}{\rm{i}}}{2}$ & $-15^3$ & $\Bbb{Z}[\frac{1+\sqrt{-7}}{2}]$\\
\hline
$2{\rm{i}}$ & $17\pm 12\sqrt{2}$ & $66^3$ &$\Bbb{Z}[2\sqrt{-1}]$ \\
\hline
$2\sqrt{2}{\rm{i}}$ & $\frac{1}{4}\left(\sqrt{2\sqrt{2}-2}+\frac{\sqrt{2\sqrt{2}+2}}{2}+2\right)$ & $10^3\left(5+\sqrt{2}\right)^{3}\left(7+5\sqrt{2}\right)^2$ & $\Bbb{Z}[\sqrt{-8}]$\\ 
\hline
$2\sqrt{3}{\rm{i}}$ & $\frac{1}{4}\left(2\sqrt{2-\sqrt{3}}+\frac{\sqrt{2+\sqrt{3}}}{2}+2\right)$ &  $4\times 15^3\left(6-\sqrt{3}\right)^3\left(26+15\sqrt{3}\right)^2$ & $\Bbb{Z}[\sqrt{-12}]$\\
\hline
$\sqrt{7}{\rm{i}}$ & $\frac{1}{4}\left(\frac{\sqrt{62+ 6\sqrt{7}{\rm{i}}}}{8}+\frac{\sqrt{62- 6\sqrt{7}{\rm{i}}}}{8}+2\right)$ & $255^3$ & $\Bbb{Z}[\sqrt{-7}]$\\
\hline
$4{\rm{i}}$ &  $\frac{1}{4}\left(2\sqrt{3\sqrt{2}-4}+\frac{\sqrt{6\sqrt{2}+8}}{4}+2\right)$ & $\frac{1}{2}\left(99\sqrt{2}-12\right)^3\left(99\sqrt{2}+140\right)^2$ & $\Bbb{Z}[4\sqrt{-1}]$\\
\hline
$\frac{-1+2\sqrt{2}{\rm{i}}}{3}$ & $\frac{1}{4}\left({\rm{i}}\sqrt{2\sqrt{2}+2}-{\rm{i}}\frac{\sqrt{2\sqrt{2}-2}}{2}+2\right)$  & $10^3\left(5-\sqrt{2}\right)^{3}\left(7-5\sqrt{2}\right)^2$ & $\Bbb{Z}[2\sqrt{-2}]$\\
\hline
$3{\rm{i}}$ &  $-\left(2+\sqrt{3}\right)^2\left(\sqrt{2}+\sqrt[4]{3}\right)^4$ & $64\left(387+224\sqrt{3}\right)^3\left(97-56\sqrt{3}\right)$ & $\Bbb{Z}[3\sqrt{-1}]$\\
\hline
$\frac{1+3{\rm{i}}}{2}$ & $-\left(2-\sqrt{3}\right)^2\left(\sqrt{2}-{\rm{i}}\sqrt[4]{3}\right)^4$ & $64\left(387-224\sqrt{3}\right)^3\left(97+56\sqrt{3}\right)$ & $\Bbb{Z}[3\sqrt{-1}]$\\
\hline
$\frac{1+\sqrt{15}{\rm{i}}}{4}$ & $\left(\frac{\sqrt{5}-1}{8}(1\pm\left(2\sqrt{3}+\sqrt{15}){\rm{i}}\right)\right)^4$ & $\frac{(7-3\sqrt{5})^2 (-1+3\sqrt{5})^3 (15+3\sqrt{5})^3}{256}$ & $\Bbb{Z}[\frac{1+\sqrt{-15}}{2}]$\\
\hline
$\frac{1+\sqrt{15}{\rm{i}}}{2}$ & $\left(\frac{\sqrt{5}+1}{8}(-1\pm\left(2\sqrt{3}-\sqrt{15}\right){\rm{i}})\right)^4$ &  $\frac{(7+3\sqrt{5})^2 (-1-3\sqrt{5})^3 (15-3\sqrt{5})^3}{256}$ & $\Bbb{Z}[\frac{1+\sqrt{-15}}{2}]$\\
\hline
\end{tabular}
\label{tab:1}
\end{table}
\end{center}

\vskip 0.4 cm
{\bf{Acknowledgement.}} The author would like to thank Professors H. Fanai, A. Kamalinejad and M. Shahshahani for their consistent encouragement, and A. K. and M. S. for making \cite{cwd} available to him, and very helpful discussions on various topics. I am also grateful to Professors A. Rajaei and R. Takloo-Bighash for the time they spent reading this manuscript and their valuable suggestions.

%%%%%%%%%%%%%%%%%%%%%%%%%%%%%%%%%%%%%%%%%%%%%%%%%%%%%%%%%%%%%%%%%%%%%%%%%%%%%%%%%%%%%%%%%%%%%%%%%%%%%%%%

\vskip .5 cm
\begin{thebibliography}{99}
\bibitem{cwd} A. Kamalinejad and M. Shahshahani , {\it{On Computations with Dessins d'Enfants}}, to appear in {\it Mathematics of Computation}. 
\bibitem{Cohen} H. Cohen, {\it A Course in Computational Number Theory}, Springer-Verlag, (1993).
\end{thebibliography}
\end{document}